\documentclass[11pt,a4paper,reqno]{amsart} 

\usepackage{amsmath,amssymb,tikz,color}
\usepackage{subfigure}
\usepackage{graphicx}
\usepackage{mathtools}
\usetikzlibrary{shapes.geometric, calc}
\usepackage{bbm}

 \newtheorem{theorem}{Theorem}[section]

 \newtheorem{lemma}[theorem]{Lemma}

  \allowdisplaybreaks

 \theoremstyle{definition}

 \newtheorem{definition}[theorem]{Definition}
 
 \newtheorem{remark}[theorem]{Remark}

\DeclarePairedDelimiter\floor{\lfloor}{\rfloor}

\newcommand{\T}{\mathcal{T}}
\newcommand{\Loop}{\mathsf{Loop}}
\newcommand{\Height}{\mathsf{Height}}

\renewcommand{\deg}{\mathsf{deg}}
\newcommand{\out}{\mathsf{out}}
\newcommand{\sd}{\mathrm{\mathbf{SD}}}

\newcommand{\eqd}{ \,{\buildrel d \over =}\ }

\newcommand{\cstar}{\star}
\newcommand{\csharp}{\#}
\newcommand{\cond}{C}
\newcommand{\GW}{\operatorname{GW}}
\newcommand{\diam}{\operatorname{diam}}

\title[On scaling limits of random Halin-like maps]{On scaling limits of random Halin-like maps}
\author[D. Amankwah]{Daniel Amankwah}
\address{Mathematics Division, Science Institute, University of Iceland, Dunhaga 5, 107 Reykjav\'ik, Iceland.} \email{daa33@hi.is}
\author[S.\"O. Stef\'ansson]{Sigurdur \"Orn Stef\'ansson}
\email{sigurdur@hi.is}

\date{\today}

\begin{document}

\begin{abstract}
We consider maps which are constructed from plane trees by assigning marks to the corners of each vertex and then connecting each pair of consecutive marks on their contour by a single edge.  A measure is defined on the set of such maps by assigning Boltzmann weights to the faces. When every vertex has exactly one marked corner, these maps are dissections of a polygon which are bijectively related to non-crossing trees. When every vertex has at least one marked corner, the maps are outerplanar and each of its two-connected component is bijectively related to a non-crossing tree. We study the scaling limits of the maps under these conditions and establish that for certain choices of the weights the scaling limits are either the Brownian CRT or the $\alpha$-stable looptrees of Curien and Kortchemski.

\end{abstract}

\maketitle

\section{Introduction}\label{sec:intro}

A planar map is an embedding of a planar graph onto a 2-dimensional sphere up to continuous deformations. We refer to a planar map as rooted if it has a distinguished and oriented edge. Rooting a planar map makes it easier to handle combinatorially since it eliminates symmetry issues. Also, the root gives a starting point for recursive decomposition. The tail of the root edge will be called the root vertex. The connected components of the complement of a planar map are called faces. 
If $f$ is a face of a planar map $M$, we will denote its degree (the number of edge sides that surround it, counted with multiplicity) by $\deg_M(f)$ or just $\deg(f)$ when there is no ambiguity. A rooted map with a single face is called a plane tree.

Research into asymptotic behavior of random planar maps has been active in the last two decades. This has been motivated in part by the connection with two-dimensional Liouville quantum gravity(see \cite{LeGall2,Miermont} together with references for detailed account). Much of the advances in this field have been made possible by the use of bijections between specific classes of planar maps and certain plane trees. See for example \cite{CoriVau,Schaef,BouFranGui}. This is due to the fact that plane trees are easier to analyse since they are coded by functions on which one can apply the standard tools from analysis and probability theory. In \cite{Aldous1,Aldous2,Aldous3}, Aldous gives an exposition on the scaling limit of various classes of discrete plane trees conditioned to be large. For a Galton-Watson tree conditioned to be large, whose offspring distribution has mean $1$ and a finite variance, he showed that the scaling limit, in the Gromov-Hausdorff sense (see Section \ref{ss:GH}) is a continuous random tree called the Brownian continuum random tree (CRT). The CRT is often denoted by $\T_{\mathbf{e}}$ where $\mathbf{e}$ is a standard Brownian excursion of duration 1 from which the tree may be constructed. It admits a universality property in the sense that it is the limit of several classes of trees and maps including dissections of polygons and outerplanar maps, which are of interest in this work. A dissection of a polygon is a two-connected planar map, all of whose vertices lie on the boundary of the unbounded face. Outerplanar maps are maps whose two-connected components are dissections. Curien, Haas and Kortchemski \cite{randissection} established such limit theorems in the case of random dissections and this was first shown by Caraceni \cite{caraceni} in the case of uniform random outerplanar maps and later generalized by Stufler \cite{stufler:2020}. 

Generalizations to Galton-Watson trees whose offspring disribution is in the domain of attraction of a stable law with index $\alpha\in(1,2]$ were given by Duquesne and Le Gall (see \cite{DuquesneGW,DuqGalllevypro}). They showed that these trees, with a properly rescaled graph distance admit a weak limit which is referred to as the $\alpha$-stable tree, $\mathcal{T}_\alpha$. In the case $\alpha=2$ we recover the Brownian CRT, more precisely $\T_2 = \sqrt{2}\T_\mathbf{e}$. The $\alpha$-stable trees are qualitatively different from the Brownian CRT when $\alpha < 2$, since then they possess a countably infinite number of vertices of infinite degree. 

Curien and Kortchemski introduced the study of a family of random compact metric spaces $(\mathcal{L}_{\alpha})_{1<\alpha<2}$ called $\alpha$-stable looptrees (see  \cite{CurienKortchloop}). On the discrete level, looptrees are graphs constructed on the vertex set of a plane tree by connecting each vertex to its leftmost child by an edge, connecting adjacent siblings by edges, and finally connecting the rightmost child to its parent by an edge (see Fig.~\ref{looptree}).  The space $\mathcal{L}_\alpha$ may be informally described in the same spirit by replacing each vertex of large degree in the $\alpha$-stable tree $\mathcal{T}_\alpha$, by a circle of length proportional to this degree.
\begin{figure}[ht]
	\centering
	
	\resizebox{0.5\textwidth}{!}{
	\includegraphics{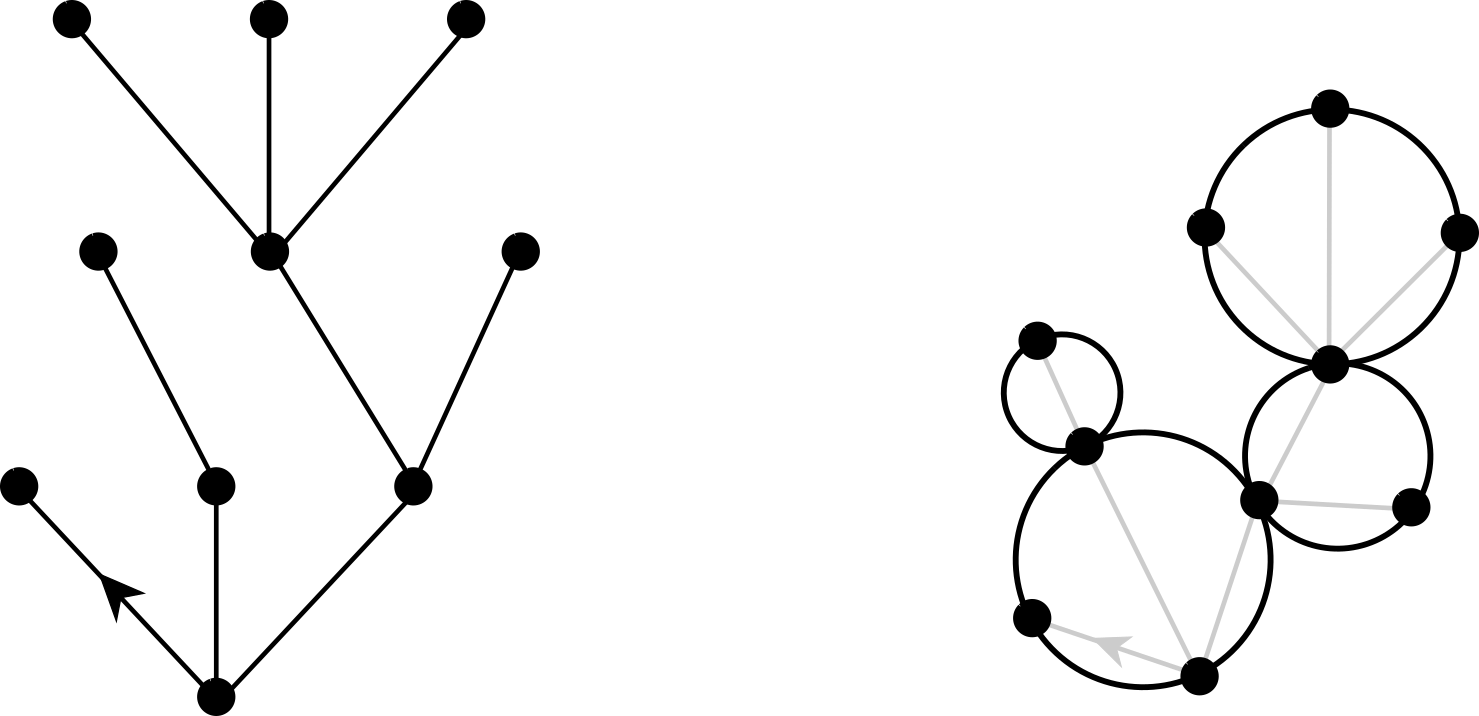}
		}  
	\caption{A looptree associated to a plane tree.}
	\label{looptree}
\end{figure}     
Curien and Kortchemski gave sufficient conditions for which a random sequence of discrete looptrees converges in distribution to the stable looptree. As an application they showed that the random dissection of a polygon sampled according to a Boltzmann distribution on the face degrees obeying the relevant conditions, converges, once properly rescaled, to the stable looptree for the Gromov-Hausdorff topology. Furthermore, Stefánsson and Stufler used this invariance principle to show that under certain conditions, properly rescaled Boltzmann face-weighted outerplanar maps also converge to stable looptrees \cite{stefansson:2019}. 

These examples have the common feature that the objects involved are tree-like and when their faces are forced to have large degrees, in the limit, they appear as circles arranged along a tree structure. We are interested in understanding how generic this behavior is and provide in this paper additional examples of families of maps which share this feature of being tree-like and converging to the Brownian CRT or $\alpha$-stable looptrees.

\subsection{Halin-like maps}

We will assume (using the stereographic projection) that our planar maps are drawn in the infinite plane. The maps then have exactly one unbounded face which will be denoted by $f_\infty$. The  maps we consider in this work are defined as follows:
 Start with a rooted plane tree and assign marks to some of the corners around its vertices. Denote this marked tree by $H^\circ$. Then go around the contour of the tree and connect each marked corner with an edge to the next marked corner on the contour. Denote the resulting map by $H$. We call $H^\circ$ the \emph{skeleton} of $H$. We use the convention that the root of $H$ lies on the boundary of the unbounded face and is the directed edge from the first to the second marked corner encountered in a clockwise exploration around $H^\circ$, starting from the root vertex. The face containing the root edge is called the root face and is denoted by $f_0$. 
 
 In the special case when the skeleton $H^\circ$ has at least four edges and no vertex of degree two, and all its leaves but no other vertices receive a mark, the corresponding graph $H$ has been referred to as a Halin graph in the literature, see e.g.~\cite{han} and Fig.~\ref{fig:actualhalin}. 
  \begin{figure}[ht]
 	\centering
 	\resizebox{0.75\textwidth}{!}{
 		\includegraphics[width=30cm]{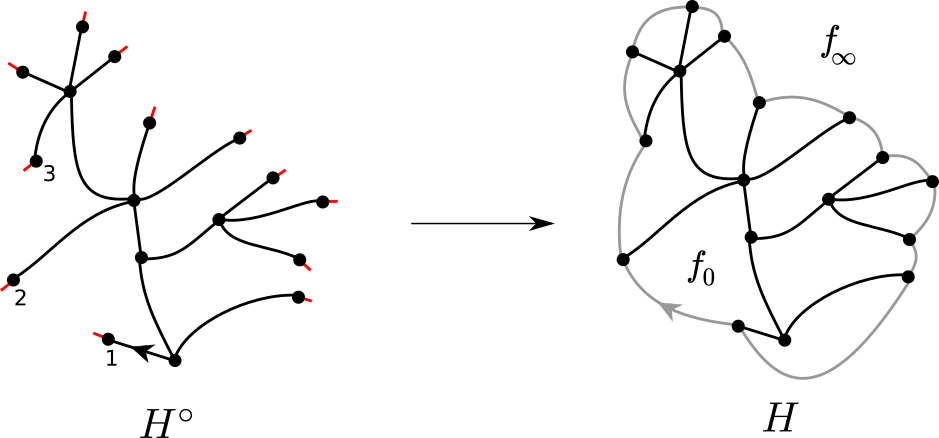}
 	}  
 	\caption{The skeleton tree $H^\circ$ has no vertices of degree two and each leaf has a mark but not other vertices. The corresponding map $H$ is then a Halin graph. The marks on corners in $H^\circ$ are indicated by red line segments and the root edges are indciated by arrows. \label{fig:actualhalin}}
 \end{figure}
 Due to this connection we refer to the maps considered in the current work as \emph{Halin-like maps}. There are other models which bear similarity with Halin-like maps such as graphs with a given surplus of edges which appear as components in the critical Erdös-Rényi random graph \cite{BerryBroutinGoldschmidt} and unicellular maps, which are generalizations of trees on surfaces of arbitrary genus, see e.g.~\cite{Chapuy}.
 
 In the current work we will consider separately two special cases in which we put the following assumption on the Halin-like maps which are best described as conditions on the marks of the skeleton tree:

\medskip
\begin{center}
	\framebox{
	\begin{minipage}{0.9\linewidth}
	($\cond^{\cstar}$) \quad  Each vertex in the skeleton tree $H^\circ$ of the  map $H$ has \textbf{exactly} one marked corner. The first corner counterclockwise from the root edge is marked.
\end{minipage}}
\end{center}
\medskip
\begin{center}
	\framebox{
		\begin{minipage}{0.9\linewidth}
			($\cond^{\csharp}$) \quad  Each vertex in the skeleton tree $H^\circ$ of the map $H$ has \textbf{at least} one marked corner. The first corner counterclockwise from the root edge is marked.
	\end{minipage}}
\end{center}

\medskip

Under condition $\cond^\cstar$, the skeleton tree $H^\circ$ is a so-called \emph{non-crossing tree} (NCT), see e.g.~\cite{Noncrossing} in the context of scaling limits. The corresponding map $H$ is then a dissection of a polygon which we will refer to as a \emph{non-crossing dissection} (NCD), see Fig.~\ref{fig:contractionstar}. Under condition $\cond^\csharp$, the map $H$ is an outerplanar map whose blocks are NCDs. We will refer to such an outerplanar map as a non-crossing outerplanar map (NCO), see Fig.~\ref{fig:HalGrEx1}.

\begin{figure}[ht]
	\centering
	\resizebox{0.75\textwidth}{!}{
		\includegraphics[width=30cm]{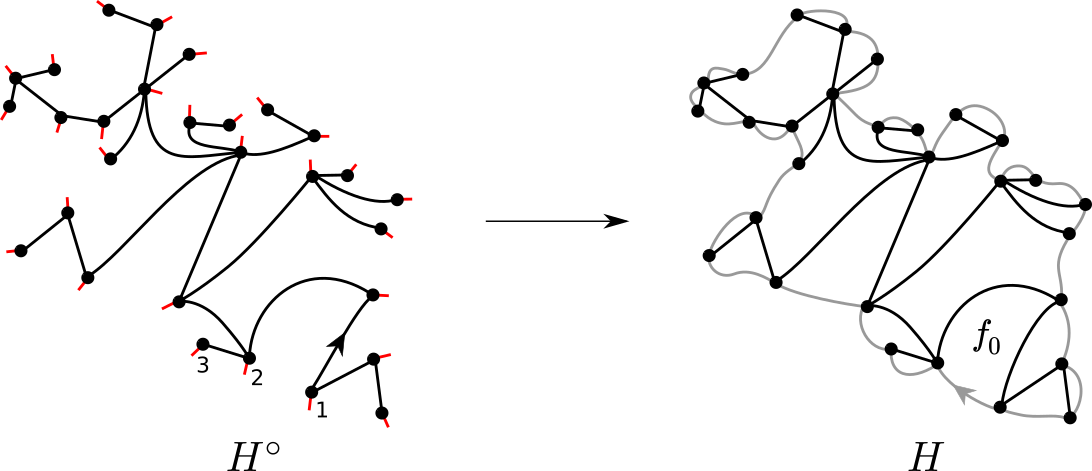}
	}  
	\caption{The skeleton tree $H^\circ$ satisfying condition $\cond^\cstar$ and its corresponding map $H$. The tree is a non-crossing tree and the map is a non-crossing dissection.  \label{fig:contractionstar}}
\end{figure}

\begin{figure}[ht]
	\centering
	\resizebox{.75\textwidth}{!}{
		\includegraphics[width=30cm]{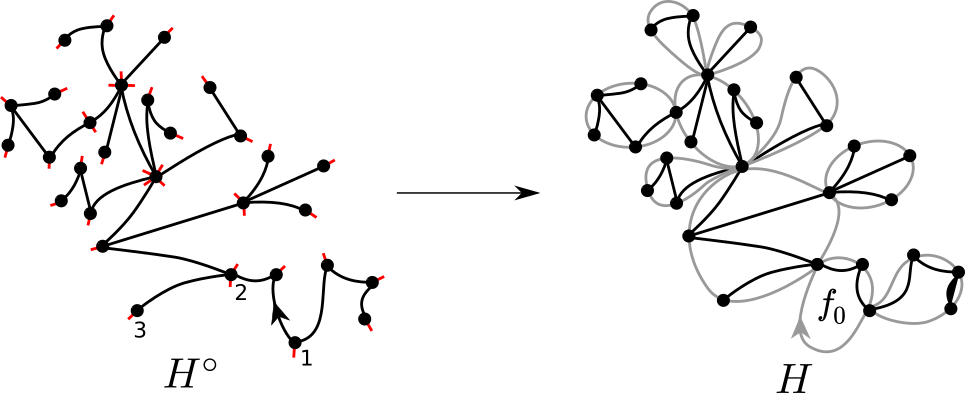}
	}  
	\caption{ The skeleton tree $H^\circ$ satisfying condition $\cond^\csharp$ and its corresponding map $H$ which is a non-crossing outerplanar map. \label{fig:HalGrEx1}}
\end{figure}

\subsection{The main result}\label{sec:result_intro}
For every integer $n\geq 1$, denote by $\mathbb{H}^{\ast}_{n}$ the set of Halin-like maps satisfying $(\cond^\ast)$,  $\ast \in \{\cstar,\csharp\}$ and having $n$ bounded faces. When it is clear under which assumption we are working, we will leave out the $\ast$ from the notation.  Let $H \in \mathbb{H}_n^\ast$  and fix a sequence $w=(w(k):k \geq 2)$ of non-negative \emph{face weights}. We define the weight of the map $H$ by 

\begin{equation*}
W(H) = \prod_{\substack{f \; \text{face of} \; H \\ f\neq f_\infty}}w(\deg(f))
\end{equation*} 
and the normalization
\begin{equation*}
Z_{n}^\ast = \sum_{H' \in \mathbb{H}^\ast_{n}}W(H').
\end{equation*}
If $Z^\ast_{n}>0$, we define a random element $\mathcal{H}_n^{w,\ast}$ in $\mathbb{H}^\ast_n$ by
\begin{equation}
\mathbb{P}(\mathcal{H}_n^{w,\ast} = H) = \frac{1}{Z^\ast_{n}} W(H) \hspace{1cm} \text{for all} \; H \in \mathbb{H}^\ast_{n}.
\end{equation}
 We view $\mathcal{H}_{n}^{w,\ast}$ as a random metric space equipped with the graph distance as its metric (where every edge of $\mathcal{H}_{n}^{w,\ast}$ has unit length).\\

It will prove convenient to define the two measures $\mu$ and $\mu_\varnothing$ on $\mathbb{Z}_+ = \{0,1,2,\ldots\}$ by
 \begin{align} \label{eq:mus}
	\mu(k) = a(k+1)w(k+2)b^k, ~k\geq 0 ~ \text{and}~ \mu_\varnothing(k) = c w(k+1) b^k,~¨ k\geq 1
\end{align}
where $a,b,c > 0$. We will assume throughout this paper that the generating function of the sequence $w$ has a non-zero radius of convergence
\begin{align*}
\rho_w = \left(\limsup_{k\to \infty} w(k)^{1/k})\right)^{-1}.
\end{align*} 
This will allow us to define for each $b \in (0,\rho_w)$
\begin{align} \label{eq:normalization}
a = \left(\sum_{k=0}^\infty (k+1)w(k+2)b^k\right)^{-1} \qquad \text{and} \qquad c =  \left(\sum_{k=1}^\infty w(k+1)b^k\right)^{-1}
\end{align}
so that $\mu$ and $\mu_\varnothing$ are probability measures. The motivation for defining the measures $\mu$ and $\mu_\varnothing$ is that they are equivalent (in the sense of Equation \eqref{eq:tilt1}) to offspring distributions of non-root vertices and the root resp.~in the weak dual tree of the Halin-like map under condition $(\cond^\cstar$). This will allow us to use the theory of branching processes to understand the distribution of the maps.  The shift by 2 in $\mu$ and 1 in $\mu_\varnothing$ is due to the deletion of 1 edge when taking the weak dual and substraction of 1 for non-root vertices since we consider only the number of offspring. The factor $k+1$ in $\mu$ comes from the fact that we can choose a location for the edge which was deleted when taking the weak dual in $k+1$ different ways. 

When $\mu$ and $\mu_\varnothing$ are probability measures, denote their mean and variance by $m_\mu$ and $\sigma_\mu^2$ respectively.  In order to guarantee that $Z_n^\cstar > 0$ for sufficiently large $n$ we assume that $\mu(k)$ is aperiodic. The general case may also be considered but the results are essentially the same and we aim to keep the notation simple.


The main results of the current work are the following four scaling limit theorems for $\mathcal{H}_{n}^{w, \ast}$. Theorems \ref{mainresultCRT1} and \ref{mainresultCRT2} show that if the faces are not forced to be too large,  $n^{-1/2}\cdot \mathcal{H}_{n}^{w,\ast}$ converges to a constant multiple (depending on $\ast$) of the Brownian CRT $\T_\mathbf{e}$. The notation $s \cdot \mathcal{H}_{n}^{w}$ represents the metric space obtained from $\mathcal{H}_{n}^{w}$ by multiplying all distances by $s>0$. Theorems \ref{mainresult1} and \ref{mainresult2} however show that it is possible to choose the weights $w$ such that faces in the maps become large when $n$ is large and the maps, properly rescaled, converge in distribution towards stable looptrees. We use techniques from four different references in the proof of the theorems. Note that the convergence is understood as the weak convergence of measures in the Gromov-Hausdorff topology on the space of compact metric spaces (see Section \ref{ss:GH} for the definition).

In the next two Theorems we assume that condition $(\cond^\cstar)$ holds and that $\mu$ and $\mu_\varnothing$ are probability measures on $\mathbb{Z}_{+}$ defined by \eqref{eq:mus} and \eqref{eq:normalization}.

\begin{theorem}\label{mainresultCRT1}
		Assume condition $(\cond^\cstar)$. If there is a $b \in (0,\rho_w]$ and a $\lambda > 0$ such that $m_\mu= 1$ and $\sum_{k=0}^\infty e^{\lambda k}\mu(k) < \infty$ then there is a constant $c^\cstar(\mu)>0$ such that 
		\begin{align*}
		n^{-1/2} \cdot  \mathcal{H}_{n}^{w,\cstar}~\xrightarrow[n \rightarrow \infty]{\enskip (d) \enskip}~c^\cstar(\mu)\cdot \mathcal{T}_\mathbf{e}.
		\end{align*}
\end{theorem}
 The proof, which is given in Section \ref{ss:proof1}, follows from the Markov-chain argument developed in \cite{randissection}, which shows that the ratio of lengths of geoedesics in large NCTs and the corresponding paths in their weak dual tree is given by
 \begin{align*}
c_\mathrm{geo}(\mu) &= \frac{1}{4}\left(\sigma_\mu^2 + 1 + \frac{\gamma(2\mathbb{Z}_++1)(1-2\gamma(\mathbb{Z}))}{1-\gamma(2\mathbb{Z}_+)+\gamma(2\mathbb{Z}_+ +1)}\right)
 \end{align*}
where  $\gamma(j) = \frac{j\mu(j)}{j+1}$ and for any $A\subseteq \mathbb{Z}_+ = \{0,1,\ldots\}$,
$\gamma(A) = \sum_{j\in A} \gamma(j).$
 By the results of Aldous,  the weak dual tree of $\mathcal{H}_n^{w,\star}$, rescaled by $n^{-1/2}$, converges towards a multiple $c_{\mathrm{tree}}(\mu) := 2\sigma_\mu^{-1}$ of the Brownian CRT. Combining these results yields $c^\cstar(\mu) = c_{\mathrm{tree}}(\mu)\cdot c_{\mathrm{geo}}(\mu).$

It should be possible to relax the exponential moment condition on $\mu$ and only assume finite variance, see e.g.~\cite{kortchemskirichier}, but we do not pursue this further.

\begin{theorem}\label{mainresult1}
		Assume condition $(\cond^\cstar)$. If there is a $b \in (0,\rho_w]$ such that $m_\mu= 1$ and $\mu$ belongs to the domain of attraction of an $\alpha$-stable law then there is a sequence of positive numbers $b_n\to \infty$ such that
	$$b_n^{-1}\cdot \mathcal{H}_{n}^{w,\cstar}~\xrightarrow[n \rightarrow \infty]{\enskip (d) \enskip}~\mathcal{L}_{\alpha}.$$
\end{theorem}
A probability measure $\mu$ is said to be in the domain of attraction of an $\alpha$-stable law, with $\alpha\in(1,2)$, if $$\mu([j,\infty))=j^{-\alpha}L(j),$$ where $L:\mathbb{R}_{+}\rightarrow \mathbb{R}_{+}$ is a slowly varying function, i.e.~$L(x)>0$ and $\lim_{x\rightarrow \infty}\frac{L(tx)}{L(x)}=1$ for all $t>0$ (see \cite{Slowly varying} for details). In this case, the variance of $\mu$ is infinite. Common examples of slowly varying functions are constant functions, powers of logarithms, and iterated logarithms. By \cite[Theorem 1.10]{Kortchinva}, it holds that 
\begin{align} \label{eq:bscaling}
	b_n = n^{1/\alpha} \Lambda(n)
\end{align}
where $\Lambda$ is a slowly varying function. An explicit example of a measure $\mu$ which satisfies the conditions of the theorem is 
\begin{align*}
	\mu(k) = ck^{-1-\alpha}, \qquad c>0
\end{align*}
in which case 
\begin{align*}
	b_n =  \frac{n^{\frac{1}{\alpha}}}{(c\vert\Gamma(-\alpha)\vert)^{\frac{1}{\alpha}}}.
\end{align*}
We refer the reader to e.g.~\cite{CurienKortchloop,Kortchinva} for more information on random variables in the domain of attraction of stable laws, in a relevant context. The proof of Theorem \ref{mainresult1}, which is given in Section \ref{ss:proof2}, follows directly from the invariance principle presented in \cite[Theorem 4.1]{CurienKortchloop}, along with an argument showing that the Gromov-Hausdorff distance between the NCD and the looptree of its weak dual tree is sufficiently small.

In the next two Theorems we assume that $\mu$ and $\mu_\varnothing$ are defined by \eqref{eq:mus}, but they are not necessarily probability measures. Before the statement of the Theorems we require some notation. Define the generating functions \begin{align*}
	 g^\mu(z) = \sum_{n=0}^\infty \mu(n) z^n \qquad \text{and}\qquad  g_\varnothing^\mu(z) = \sum_{n=0}^\infty \mu_\varnothing(n) z^n  
\end{align*}
and denote their radius of convergence by $\rho_\mu$. Define
\begin{align*}
	\nu_\mu = \lim_{x\nearrow \rho_\mu} \frac{x (g^\mu)'(x)}{g^\mu(x)}
\end{align*}
as the largest possible mean that $\mu$ can have, when $b$ is varied. 
Let \begin{align*}
	A_\mu = \rho_\mu \frac{g_\varnothing^\mu(\rho_\mu)}{g^\mu(\rho_\mu)} \qquad \text{and}\qquad B_\mu = \rho_\mu \frac{(g_\varnothing^\mu)'(\rho_\mu)}{g^\mu(\rho_\mu)}.
\end{align*}

\begin{theorem}\label{mainresultCRT2}
		Assume condition $(\cond^{\csharp})$. If one of the three cases
		\begin{align*}
		\begin{cases}
		\nu_\mu \geq 1, \\
		0 < \nu_\mu < 1, \rho_\mu \geq 1, \\
		0 < \nu_\mu < 1, \frac{A_\mu(1-\nu_\mu)+\rho_\mu B_\mu}{(1- A_\mu)(1-\nu_\mu)} > 1
		\end{cases}
		\end{align*}
		holds then there is a constant $c^{\csharp}>0$ such that 
		\begin{align*}
			n^{-1/2} \cdot  \mathcal{H}_{n}^{w,\csharp}~\xrightarrow[n \rightarrow \infty]{\enskip (d) \enskip}~c^{\csharp}\cdot \mathcal{T}_\mathbf{e}.
		\end{align*}
\end{theorem}
The proof, which is given in Section \ref{ss:proof3}, follows from Theorem 6.60 in \cite{stufler:2020} which holds for random decorated trees (or enriched trees), satisfying certain conditions. The NCO is viewed as a size conditioned Galton-Watson tree $\mathcal{T}$ decorated by face weighted finite sequences of NCDs. We do not provide an explicit formula for the constant $c^{\csharp}$, but it may be calculated in terms of expected distances in face weighted finite sequences of NCDs, see Section 7.11 in \cite{stufler:2020}. The formula for its value will depend on which of the above three cases holds.  

It will be apparent in the proof of Theorem \ref{mainresultCRT2} that the above conditions guarantee that the offspring distribution of $\mathcal{T}$ has finite exponential moments. We remark again that this exponential moment condition could possibly be relaxed to only assuming that the offspring distribution has finite variance. 

\begin{theorem}\label{mainresult2}
	Assume condition $(\cond^{\csharp})$. Fix $\alpha \in (1,2)$ and assume that 
	\begin{align} \label{muweights}
	\mu(k) = L(k) k^{-\alpha-1} r^{-k}
	\end{align}
	with $r = \rho_\mu <1$ and $L$ slowly varying. 	If $\nu_\mu < 1$ and
	\begin{align*}
		\frac{A_\mu(1-\nu_\mu)+\rho_\mu B_\mu}{(1-A_\mu)(1-\nu_\mu)} = 1
	\end{align*}
	Then 
	$$c_n^{-1}\cdot \mathcal{H}_{n}^{w,\csharp}~\xrightarrow[n \rightarrow \infty]{\enskip (d) \enskip}~\mathcal{L}_{\alpha}$$
	where
	\begin{align*}
		c_n = \left(\frac{|\Gamma(-\alpha)|g^\mu(r) B_\mu}{(1-\nu_\mu)(1-A_\mu)}\right)^{1/\alpha}(L_n n)^{1/\alpha}.
	\end{align*}
	
\end{theorem}
The proof, which is presented in Section \ref{ss:proof4} again involves representing $\mathcal{H}_{n}^{w,\csharp}$ as a size conditioned Galton-Watson tree, decorated by face weighted finite sequences of NCDs. The result then follows from similar arguments given explicitly for outerplanar maps, in \cite{stefansson:2019}. However, we will instead apply the recent more general invariance principle for decorated trees presented in \cite[Theorem 5.1]{decoratedtrees}.

\begin{remark} We have not been able to relax conditions  $(\cond^\cstar)$ and $(\cond^{\csharp})$ in the current work. If a vertex has for example no marked corners it will not be on the boundary of $f_\infty$ in the Halin-like map and we lose the connection to dissections and outer-planar maps. There are certain exceptions to this, e.g.~when $H$ is a Halin-map which satisfies that each internal vertex is connected to a leaf. Then only the leaves are marked and by contracting the edges containing the leaves one arrives again at a Halin-like map which is outerplanar. One may show that this contraction does not affect the Gromov-Hausdorff distance too much and thus one gets similar invariance principles as above. 
	\end {remark}


\section{Halin-like maps and equivalent representations by trees}\label{sec:Def_Not}
In this section we give formal definitions of plane trees and Halin-like maps and show that when they satisfy condition $(\cond^\cstar)$ or $(\cond^\csharp)$, they are equivalent to plane trees or enriched trees respectively. These connections will allow us to relate them to size conditioned Galton-Watson trees in a suitable sense which gives access to known invariance principles.

\subsection{Rooted plane trees and Halin maps}\label{sec:BoltzHalGr}

We briefly recall the formalism of rooted plane trees (or rooted ordered trees). Let $\mathbb{N}= \{1,2,...\}$ and define the set of finite sequences by 
\[ \mathcal{U}=\bigcup_{n=0}^{\infty} \mathbb{N}^{n},\]
 where by convention $\mathbb{N}^{0}=\{\varnothing\}$. An element of $\mathcal{U}$ is thus a sequence $u=(u^{1},...,u^{n})$ of elements of $\mathbb{N}$. 
 If $u=(u^{1},...,u^{m})$ and $v=(v^{1},...,v^{n})$ belong to $\mathcal{U}$, we write the concatenation of $u$ and $v$ as $uv=(u^{1},...,u^{m},v^{1},...,v^{n})$. \\
 
 A plane tree $T$ is a finite or infinite subset of $\mathcal{U}$ such that,
 \begin{enumerate}
 \item $\varnothing \in T$.
 \item if $ v \in T$ and $v=uj$ for some $j \in \mathbb{N}$ then $u \in T$.
 \item for every $u \in T $, there exists an integer $\out_T(u) \geq 0$ (the outdegree or number of children of $u$) such that for every $j \in \mathbb{N}$, $uj \in T$ if and only if $1 \leq j \leq \out_T(u)$.  
 \end{enumerate}
 The vertex $\varnothing$ is called the root vertex and the edge $\{\varnothing,(1)\}$ is called the root edge. The letter $T$ will sometimes be suppressed in $\out_T(u)$ when it is clear from the context which tree is being referred to. 
 Let us denote the set of all plane trees by $\textbf{T}$ and the set of plane trees with $n$ vertices by $\textbf{T}_{n}$. For $T \in \textbf{T}$, we will view each of its vertices as an individual of a population whose $T$ is the geneological tree. If $T \in \textbf{T}$ and $u \in T$, we define the shift of $T$ at $u$ by $\mathsf{shift}_{u}T =\{v \in \mathcal{U} : uv \in T \} $ which is itself a tree. The total progeny of $T$ (total number of vertices of $T$) will be denoted by $|T|$. The maximal generation of the tree $T$ is called its height and we denote it by $\Height(T)$.
 
Given a plane tree $T$ with $k$ edges, define the contour sequence of $T$ 
\begin{align*}
	c = (c_0,c_1,\ldots,c_{2k})
\end{align*}
as the sequence of vertices encountered when moving along the contour of $T$ in a clockwise direction, starting and ending at $c_0 = c_{2k} = \varnothing$. The contour sequence is in a one-to-one correspondence with the corners around the vertices in $T$, and we use the convention that the corner clockwise from the root vertex is $c_{0}$. 

We define marks on the corners of the tree $T$ by specifying a sequence $m = (m_0,\ldots,m_{2k})$ of 0's and 1's. The corner $c_i$ receives a mark if $m_i = 1$ but otherwise it does not. Let $|m| = \sum_i m_i$ denote the number of marks.

\begin{definition}\label{HalinDef}
Let $(T,m)$ be a rooted plane tree $T$ with corners marked by a sequence $m$. Let $v_1,\cdots,v_{|m|}$ be the sequence of marked corners in the same order as they appear in the contour sequence.  A Halin-like map $H$, is a map constructed by adding to $T$ the edges $\{v_i,v_{i+1}\}$, $1\leq i \leq |m|-1$, and the edge $\{v_{|m|},v_1\}$.  We refer to $(T,m)$ as the \emph{skeleton tree} of $H$ and use the notation $H^\circ := (T,m)$.  The root edge of $H$ is the directed edge $(v_{1},v_2)$ and the root face of $H$ is the bounded face containing the root edge. The tree $T$ is referred to as the shape of the marked tree $(T,m)$ and will be denoted by $\mathcal{S}(T,m) = T$. 
\end{definition} 
The Halin maps considered in this work are assumed to satisfy condition $(\cond^{\cstar})$ or $(\cond^{\csharp})$ in the introduction. The set of Halin maps with $n$ bounded faces, satisfying condition $C^\ast$, is denoted by $\mathbb{H}_n^\ast$, $\ast \in \{\cstar,\csharp \}$.

\subsection{Condition $(\cond^\cstar)$ and non-crossing trees}\label{sec:bij}
Under condition $(\cond^\cstar)$ the skeleton $H^\circ$ is a non-crossing tree and the corresponding map is a non-crossing dissection. We denote the set of non-crossing trees with $n$ vertices by $\mathbf{NCT}_n$ and the set of non-crossing dissections with $n$ vertices by $\mathbf{NCD}_n$.  There is another bijection between NCDs and NCTs which is defined by taking the weak dual of the NCDs: First, let $\hat H$ be the usual geometric dual and if $f$ is a face in $H$ let $\hat{f}$ be its corresponding vertex in the dual. 
The weak dual, denoted by $\bar H$, is obtained from $\hat H$ by removing from it the vertex $\hat{f}_\infty$ and all the edges $\{\cdot,\hat{f}_\infty\}$ containing it. Since $H$ is a dissection of a polygon, $\bar H$ is clearly a tree. We define its root edge as the leftmost edge emanating from the vertex $\hat f_0$ dual to the root face.  Moreover, each vertex in $\bar H$, has a unique corner from which an edge $\{\cdot,\hat{f}_\infty\}$ was removed and therefore we may view the weak dual as a non-crossing tree, see Fig.~\ref{fig:augmented}. We have thus arrived at the following
\begin{lemma} \label{l:firstbijection}
The function $\phi_n: \mathbb{H}^\cstar_n \to \textbf{NCT}_{n}$, $H \mapsto \bar H$, is a bijection. If $f$ is a face in $H \in \mathbb{H}^\cstar_n$ and $u$ the corresponding vertex in $\bar H$ then 
\begin{align*}
	\out_{\bar H}(u) = \begin{cases}
		\deg_{H}(f)-1	& \text{if}~u = \varnothing, \\
		\deg_{H}(f)-2	& \text{otherwise}.
	\end{cases}
\end{align*}
\end{lemma}
 The reason for the shift in degrees is that $\out(u)$ is an outdegree (accounting for $-1$ for each non-root vertex) and the deletion of edges when taking the weak dual (accounting for $-1$ for all vertices). The bijection $\phi_n$ furthermore defines a bijection between the set of dissections and trees of arbitrary finite sizes which we will denote by $\phi$. 
   \begin{figure}[ht]
 	\centering
 	\resizebox{1.\textwidth}{!}{
 		\includegraphics[width=0.15cm]{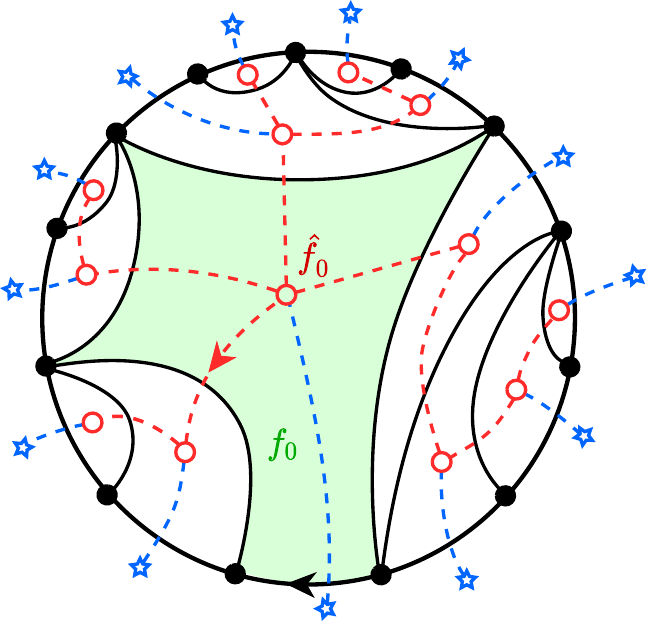}
 	}  
 	\caption{A NCD in black and its weak dual tree in dashed red. The root face is depicted in green. The augmented tree is obtained by adding the dashed blue edges.  \label{fig:augmented}}
 \end{figure}
 
 Later we will require a slightly modified version of the weak dual which is obtained by adding back to $\bar{H}$ the edges dual to the boundary of $H$ but in such a way that each of them gets a separate endpoint in $f_\infty$. This is the same notion of a dual as is used in \cite{randissection}. Equivalently, given a $T \in \mathbf{NCT}_n$, add edges to each of its marked corners which have one of their endpoints at the corresponding corner and the other one at a new leaf, see Fig.~\ref{fig:augmented}. We will call these trees \emph{augmented} NCTs and denote them by $T^\bullet$. These play the role of the \emph{planted trees} used in \cite{randissection}.

\subsection{Condition $(\cond^\csharp)$ and  enriched trees.} \label{ss:enriched} Under condition $(\cond^\csharp)$ the Halin-maps still have the property that each vertex lies on the boundary of the face $f_\infty$, however they are not necessarily two-connected. Therefore, they will in general be outerplanar maps. An outerplanar map may be decomposed into an underlying plane tree whose vertices are decorated by sequences of dissections. This representation is explained in detail in \cite[Lemma 4.3]{stefansson:2019} in the formalism of \emph{enriched trees} (or decorated trees) where it is used to obtain scaling limit results for face weighted outerplanar maps. We will follow the same approach, but in the current case the dissections involved are non-crossing dissections (NCDs) and therefore we will need to understand some of their properties in more detail.   

We give a short description of this representation below. Let $\mathbf{NCD}'_n$ be the set of NCDs with $n$ vertices where we do not count the root vertex. Pictorially we will represent the root vertex with a $\ast$.  Now define the set of finite sequences of size $n$ by
\begin{align*}
\sd_n = \bigcup_{k\geq 0} \bigcup_{\substack{n_1+\cdots + n_k = n \\ n_i \geq 1, ~1 \leq i \leq k}}  \{ (D_{n_1},D_{n_2},\ldots,D_{n_k})~:~D_{n_i} \in \mathbf{NCD}_{n_i}', 1 \leq i \leq k\}.
\end{align*}
An element from $\sd_n$ is viewed as a connected map by identifying the root vertices of each element in the sequence. Its unbounded face is denoted by $f_\infty$. The vertices of $S\in \sd_n$ are ordered by their occurence on the boundary of the unbounded face in a clockwise direction from the root-vertex. We will denote them by $s_0(S), s_1(S),\ldots,s_n(S)$ where $s_0(S)$ is the root-vertex. The directed edge $(s_0(S),s_1(S))$ is taken to be the root vertex.

Now, let $T \in \mathbf{T}_n$ be a plane tree with $n$ vertices. For each vertex $u \in T$ let $S_u$ be an element from $\sd_{\out(u)}$. Note that if $u$ is a leaf, i.e.~$\out(u) = 0$, then $S_u$ is the empty sequence. We construct a map as follows: For each non-leaf vertex $u$ and each of its non-leaf children $ui$, identify $s_i(S_u)$ and $s_0(S_{ui})$.  The resulting object is a non-crossing outerplanar map (NCO), see Fig.~\ref{fig:decorated} for an example. We denote this map by $O$. The construction is reversible, i.e.~ $O$ may be decomposed into an underlying plane tree $T$ where each vertex $u$ is decorated by an object from $\sd_{\out(u)}$. Therefore, we have a bijection between the set of NCOs and trees decorated by finite sequenced of NCDs. We will  therefore sometimes write $O = (T,(S_u)_{u\in T})$.

\begin{figure}[ht]
	\centering
	\resizebox{1.\textwidth}{!}{
		\includegraphics[width=30cm]{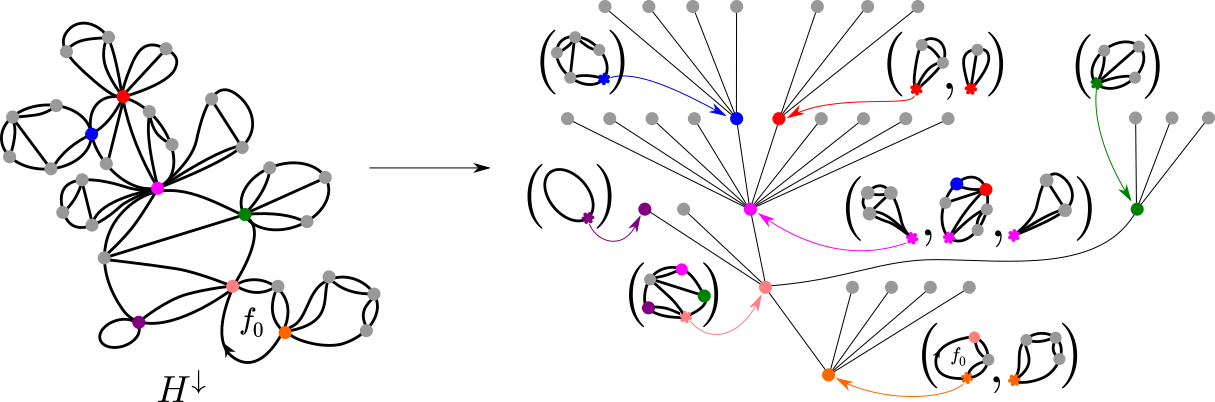}
	}  
	\caption{\textbf{Left}: A NCO, the same as in Fig.~\ref{fig:HalGrEx1}. \textbf{Right:} A corresponding decorated tree where the decorations are sequences of NCDs. Grey vertices are decorated by empty sequences. \label{fig:decorated}}
\end{figure}

\section{Simply generated trees and Galton-Watson trees}

In the previous section we showed that certain Halin-like maps may be described in terms of trees and enriched trees. In this section we show how this allows us to understand the corresponding random Halin-like maps in terms of simply generated trees and  Galton-Watson trees.

	Let $\mu$ be a probability measure on $\mathbb{Z}_{+}$ such that $\mu(1)<1$. We call $\mu$, the offspring distribution. The\textit{ Galton-Watson measure} with offspring distribution $\mu$ is the probability measure, $\GW^{\mu}$ on $\textbf{T}$ such that:
	\begin{enumerate}
		\item $\GW^{\mu}(\out({\varnothing})=j) =\mu(j)$ for $j \geq 0$,
		\item If $\tau$ is distributed by $\GW^{\mu}$, then for every $j \geq 1$ with $\mu > 0$, conditionally on $\{\out({\varnothing})=j\}$, the subtrees $\mathsf{shift}_{1}\tau, \ldots, \mathsf{shift}_{j}\tau$ are i.i.d.~with distribution $\GW^{\mu}$.
	\end{enumerate}
	Informally, the tree starts as a single root and then grows such that each individual already present in the tree has a number of offspring according to the measure $\mu$, independent of all others.

If $\mu$ has mean less than or equal to $1$ then the tree $\tau$ is almost surely finite and
\begin{equation}
\GW^{\mu}(\tau=T) = \prod_{u \in T}\mu(\out_\tau(u)).
\end{equation}
In order to relate the Galton-Watson measure to random non-crossing trees we require the following modification. Let $\mu_\varnothing$ be another offspring distribution and define 
\begin{align} \label{eq:twotype}
\GW^{\mu,\mu_\varnothing}(\tau = T) = \mu_{\varnothing}(\out_\tau(\varnothing)) \prod_{u \in T\setminus\{\varnothing\}}\mu(\out_\tau(u)).
\end{align}
We interpret this such that the root has offspring according to the measure $\mu_\varnothing$ and all other individuals have offspring according to $\mu$, independent of others.

Another related model of random trees is the following: Let $\eta = (\eta(i))_{i\geq 0}$ be a sequence of non-negative numbers, not necessarily probabilities. Let $T\in \mathbf{T}_n$ and define the weight of $T$ by
\begin{align*}
\eta(T) = \prod_{u\in T} \eta(\out_T(u)).
\end{align*}
A \emph{simply generated tree} on $n$ vertices is a random plane tree in $\mathbf{T}_n$ selected with a probability proportional to its weight $\eta(T)$. We list some important properties of simply generated trees below and refer to Janson's comprehensive survey \cite{janson} for details. Define the \emph{partition function}
\begin{align*}
Z_n^\eta = \sum_{T\in \mathbf{T}_n} \eta(T)
\end{align*}
and the generating series \begin {align} \label{partitionfunction}
\mathcal{Z}^\eta(z) = \sum_{n=1}^\infty Z^\eta_n z^n, \qquad g^\eta(z) = \sum_{n=0}^\infty \eta(n) z^n
\end{align}
and denote their radii of convergence by $R_\eta$ and $\rho_\eta$ respectively.  In the current work we will assume that $\rho_\eta > 0$ in which case it is possible, by \emph{tilting} the weights $\eta$, to assume that they are probabilities. Tilting refers to defining new weights
\begin{align}\label{eq:tilt1}
(\hat\eta(k))_{k\geq 0} = (a b^k \eta(k))_{k\geq 0}, \qquad a,b > 0
\end{align}
and noting that these define the same simply generated trees as the original weights since 
\begin{align} \label{eq:tilt2}
\hat\eta(T)= \prod_{u\in T} a b^{\out_T(u)} \eta(\out_T(u)) = a^{|T|} b^{|T|-1} \eta(T).
\end{align}

Then, the simply generated tree has the same distribution as a $\GW^\eta$ tree conditioned on having $n$ vertices which we will denote by $\GW^\eta_n$.

A standard result is the following relation between the  generating series 
\begin{align} \label{treeformula}
\mathcal{Z}^\eta(z) = z g^\eta(\mathcal{Z}^\eta(z)).
\end{align}
Define
\begin{align*}
\nu_\eta = \lim_{x\nearrow \rho_\eta} \mathbf{m}_\eta(x) \qquad \text{with}~\mathbf{m}_\eta(x) = \frac{x (g^\eta)'(x)}{g^\eta(x)}
\end{align*}
and
\begin{align} \label{taucondition}
\tau_\eta = \begin{cases}
\text{unique solution to}~\mathbf{m}_\eta(x) = 1 & \text{if}~\nu_\eta\geq 1 \\
\rho_\eta & \text{if}~\nu_\eta < 1.
\end{cases}
\end{align}
Then it holds that
\begin{align}\label{tauandR}
\tau_\eta = \mathcal{Z}^\eta(R_\eta) \qquad \text{and} \qquad R_\eta = \frac{\tau_\eta}{g^\eta(\tau^\eta)},
\end{align}
see e.g.~\cite[Remarks 7.4 and 7.5]{janson}. 
The parameter $\tau_\eta$ determines tilted weights so that the offspring distribution of the associated Galton-Watson tree has mean $\mathbf{m}_\eta(\tau_\eta)= 1$, if possible. Otherwise the mean $\mathbf{m}_\eta(\tau_\eta) = \nu_\eta < 1$ is the largest attainable. For these weights, we define their variance by
\begin{align*}
\sigma_\eta^2 = \tau_\eta \mathbf{m}'(\tau_\eta).
\end{align*}
When $\rho_\eta > 0$ we have the three cases: 

\begin{equation}\label{criteria}
\mathrm{I)} ~\nu_\eta > 1 \quad  \mathrm{II)} ~\nu_\eta = 1 \quad \mathrm{III)}~\nu_\eta < 1 
\end{equation}
in which the  large simply generated trees have different properties. In case I), the offspring distribution has finite exponential moments and the random trees, rescaled by $n^{-1/2}$, converge towards a multiple of Aldous' Brownian CRT. In case II) the trees also converge towards the CRT when $\sigma_\eta < \infty$ but when $\sigma_\eta = \infty$ one needs to assume the regularity condition that the weights are in the domain of attraction of an $\alpha$-stable law, with $\alpha \in (1,2]$. In that case, the trees, rescaled by $n^{-\frac{\alpha-1}{\alpha}}$ (up to slowly varying functions), converge towards the $\alpha$-stable tree  \cite{DuquesneGW,DuqGalllevypro}. Finally, in case III) there is no interesting scaling limit \cite{kortchemski} but there appears a unique vertex of degree $(1-\nu_\eta)n$ in the trees when they become large. This phenomenon has been referred to as \emph{condensation} \cite{janson,jonsson}. All the three cases above will be relevant when we consider the scaling limits of the Halin-like maps.

It is also possible to define modified simply generated trees which correspond to Galton-Watson trees distributed by $\GW^{\eta,\eta_\varnothing}$ conditioned on having $n$ vertices, denoted by $\GW_n^{\eta,\eta_\varnothing}$. First assume that $\eta$ and $\eta_\varnothing$ are sequences of non-negative weights, not necessarily probabilities. Define the weight of a tree $T$ by
\begin{align} \label{eq:twotype}
	\eta_\varnothing(T) = \eta_{\varnothing}(\out_T(\varnothing)) \prod_{u \in T\setminus\{\varnothing\}}\eta(\out_T(u)),
\end{align}
the partition function
\begin{align*}
	Z^\eta_{\varnothing,n} = \sum_{T\in \mathbf{T}_n} \eta_\varnothing(T)
\end{align*}
the generating series
\begin {align*}
\mathcal{Z}^\eta_\varnothing(z) = \sum_{n=1}^\infty Z_{\varnothing,n}^\eta z^n, \qquad g^\eta_\varnothing(z) = \sum_{n=0}^\infty \eta_\varnothing(n) z^n
\end{align*}
and denote their radii of convergence by $R_{\varnothing,\eta}$ and $\rho_{\varnothing,\eta}$ respectively. A modified simply generated tree with weights $\eta$ and $\eta_\varnothing$ is a random tree chosen with probability proportional to its weight $\eta_\varnothing$. The tilting in \eqref{eq:tilt1} may be applied simultaneously to $\eta$ and $\eta_\varnothing$ without changing the distribution.  We will, unless otherwise stated, assume that $\mu$ is given by \eqref{eq:mus} from which it follows that $\rho_{\varnothing,\eta}  = \rho_\eta$.

We now have the pair of equations \eqref{treeformula} and 
\begin{align} \label{treeformula2}
	\mathcal{Z}^\eta_\varnothing(z) = z g^\eta_\varnothing(\mathcal{Z}^\eta(z))
\end{align}
which immediately yields that $R_{\varnothing,\eta} = R_\eta$, since $\rho_{\varnothing,\eta} = \rho_\eta$.
We will still consider the three cases given by the criteria in \eqref{criteria} when we study the modified Galton-Watson trees. Note that these criteria do not depend on the probabilities $\eta_\varnothing$.

\subsection{Random non-crossing trees as modified Galton-Watson trees}

We will now explain how random non-crossing trees are related to the modified Galton-Watson trees distributed by $\GW^{\mu,\mu_\varnothing}$ and then we will present some results which show that these trees are in many aspects very close to $\GW^\mu$ trees. We start by stating the following result which is due to Kortchemski and Marzouk, in the context of non-crossing trees. 
\begin{theorem}[Kortchemski \& Marzouk, Theorem 18 \cite{Noncrossing}]\label{rem:G-W}
	Assume that condition $(\cond^\cstar)$ holds. Then
	\begin{align*}
	\mathcal{S}(\phi_n(\mathcal{H}_n^{w,\star})) \eqd \GW_n^{\mu,\mu_\varnothing}
	\end{align*}
	where $\mu$ and $\mu_\varnothing$ are probability measures defined by \eqref{eq:mus} and \eqref{eq:normalization}, $\phi_n$ denotes the bijection in Lemma \eqref{l:firstbijection} and $\mathcal{S}$ maps a non-crossing tree to its shape.
\end{theorem}
The proof follows from the fact that $\mathcal{H}_n^{w,\star}$ is a face weighted dissection and the outdegrees in $\phi_n(\mathcal{H}_n^{w,\star})$ are related to the face degrees in  $\mathcal{H}_n^{w,\star}$ by Lemma \ref{l:firstbijection}. 

The next lemma is analogous to \cite[Proposition 23]{Noncrossing} which applies in the case when $\mu$ is critical and is in the domain of attraction of a stable law with index $\alpha \in (1,2]$, which belongs to cases I and II in \eqref{criteria}. The current version shows that a similar result holds when $\nu_\mu < 1$ which belongs to case III in \eqref{criteria}. These results both show that when the two type trees $\GW_n^{\mu,\mu_{\varnothing}}$ are large, the degree of the root converges towards a finite random variable and that exactly one of the subtrees attached to the root has mass $M_n = n - o(1)$  with high probability. In particular, this large subtree, conditionally on $M_n$, behaves as a single type $\GW_{M_n}^\mu$ tree. This allows us to conclude that many of the asymptotic properties of the two type trees are the same as those of the single type trees. In particular, when $\nu_\mu < 1$, condensation will occur in the largest subtree of the root (but not at the root itself).

	Let $N_n$ denote the degree of the root and $M_n$ denote the size of the largest subtree of the root in a random tree distributed by $\GW^{\mu,\mu_\varnothing}_n$. Let $\mathcal{N}$ be a random variable with distribution
\begin{align*}
	\mathbb{P}(\mathcal{N}=k) = \frac{k \mu_\varnothing(k)}{\sum_{j\geq 1} j  \mu_\varnothing(j)}
\end{align*}
for $k\geq 1$. Let $(\tau_i)_{i\geq 0}$ be a sequence of independent trees distributed as $\GW^\mu$. When $\mu$ is critical or subcritical, $|\tau_i|$ is a.s.~finite. 

\begin{lemma} \label{l:cond} 
	Fix $\alpha > 1$ and let
	\begin{align} \label{muweights}
		\mu(k) = L(k) k^{-\alpha-1}
	\end{align} 
	be probabilities with mean $m = \nu_\mu < 1$ where $L$ is slowly varying. Let $\mu_\varnothing$ be the probabilities related to $\mu$ via \eqref{eq:mus}and \eqref{eq:normalization}. 
Then
\begin{align*}
N_n  \xrightarrow[n \rightarrow \infty]{\enskip (d) \enskip} \mathcal{N}
\end{align*}
and
\begin{align*}
n-1-M_n  \xrightarrow[n \rightarrow \infty]{\enskip (d) \enskip} \sum_{i=1}^{\mathcal{N}-1} |\tau_i|.
\end{align*}
\end{lemma}
\begin{proof}
This can be shown with direct but somewhat lengthy calculations but instead we use a recent general result by Stufler on Gibbs-partitions with a critical composition scheme \cite[Theorem 3.11, (i)]{stufler-gibbs-2}. We consider Equation \eqref{treeformula2} (with $\eta$ replaced by $\mu$) as such a composition scheme where $\mathcal{Z}^\mu_\varnothing$, $g^\mu_\varnothing$ and $\mathcal{Z}^\mu$ correspond respectively to $U$, $V$ and $W$ in \cite{stufler-gibbs-2}. In order to satisfy the requirements of this theorem, it suffices to have
\begin{align}
g'_\varnothing(\mathcal{Z}^\mu(R_\mu)) &< \infty, \label{case_gp}\\
	\rho_\mu &= \frac{}{} \mathcal{Z}^\mu(R_\mu),\label{cond_caseIII}\\
	\mu_\varnothing(k) &= \frac{c}{a}L_1(k) k^{-(\alpha+2)} \rho_\mu^{-k} \quad \text{and} \label{cond_muroot}\\
	Z_n^\mu &= \rho_\mu L_2(n) R_\mu^{-n}  ((1-\nu_\mu)n)^{-(\alpha+1)} \label{cond_partition}
\end{align}
for slowly varying functions $L_1$ and $L_2$.
First note that $\rho_{\varnothing,\mu} = \rho_\mu = 1$ and by \eqref{eq:mus} and \eqref{eq:normalization} it holds that
\begin{align*}
	\mu_\varnothing(k) \sim \frac{c}{a} L(k) k^{-\alpha-2}.
\end{align*}
Therefore \eqref{cond_muroot} holds with $L_1(k) \sim L(k)$ as $k\to\infty$.  Since $\nu_\mu < 1$ we are in case III in \eqref{criteria} and therefore \eqref{cond_caseIII} holds which further gives $\mathcal{Z}^\mu(R_\mu) = 1$. Equation \eqref{case_gp} then reads $g_\varnothing'(1) < \infty$ which holds due to \eqref{cond_muroot}. Finally, Equation \eqref{cond_partition} follows from \cite[Equation 14]{kortchemski} with $L_2(n) \sim L(n)$ as $n\to \infty$.

\end{proof}

\subsection{Random non-crossing outerplanar maps as random enriched trees}

Recall the definition $\sd_n$ of the finite sequences of NCDs in Section \ref{ss:enriched}. Let $w=(w(k):k \geq 2)$ be a weight sequence. Let $S \in \sd_n$ and define its weight as
\begin{align*}
 W(S) = \prod_{\substack{f \; \text{face of} \; S \\ f\neq f_\infty}}w(\deg(f))
\end{align*}
and the normalization by
\begin{align*}
	Z^{\eta,\text{seq}}_n = \sum_{S'\in\sd_n} W(S')
\end{align*}
with the convention that $Z^{\eta,\text{seq}}_0 = 1$. Let $\mathcal{SD}^w_n$ be the random element in $\sd_n$ distributed by $W(\cdot)/Z^{\eta,\text{seq}}_n$. Define the generating function
\begin{align*}
\mathcal{Z}^{\eta,\text{seq}}(z) = \sum_{n=0}^\infty Z^{\eta,\text{seq}}_n z^n.
\end{align*}
The next result shows how we may construct a random non-crossing outerplanar map by viewing it as an enriched tree. The statement is essentially the same as in Lemma 4.2 in \cite{stefansson:2019}, taking into account that the dissections are non-crossing in the current case. The proof follows directly from the representation of the NCO as a tree decorated by finite sequences of NCDs, explained in Section \ref{ss:enriched}.

\begin{lemma} \label{l:coupling}
	Assume that condition $(\cond^\csharp)$ holds. Let $\mathcal T_n^\eta$ be an $\eta$-simply generated tree, where $\eta(j) = Z^{\eta,\text{seq}}_j$ for all $j\geq 0$. Conditionally on $ \mathcal{T}_n^\eta$, for each vertex $u \in \mathcal{T}_n^\eta$, sample $\beta^w_u$ independently according the the distribution of $\mathcal{SD}^w_{\out(u)}$.  Then 
	\begin{align*}
		\mathcal{H}_n^{w,\csharp} \eqd  (\mathcal{T}_n^\eta,(\beta_u^w)_{u\in \mathcal{T}_n^\eta}).
	\end{align*}
\end{lemma}

In order to apply Lemma 3.2 to the study of scaling limits we need to understand how the weights $w$ affect the behaviour of $\beta_n^w$  and $\mathcal{T}_n^\eta$ for $n$ large. Each element in the finite sequence $\beta_n^w$ is a NCD which has a NCT as its weak dual. Thus, the possible cases which we consider for the weights $w$ will still be of the form given in \eqref{criteria} with $\eta = \mu$ and $\mu$ related to $w$ through Equation \eqref{eq:mus}. We then find the following cases for the weights $\eta$ according to the choice of $\mu$ (or equivalently $w$). We note that this result is analogous to the one for outerplanar maps obtained in \cite[Lemma 2.62]{stufler:2020}, but in the current case the dissections are non-crossing. The proof will be very similar but we include it for completeness.

\begin {lemma} \label{nuetaconditions} Let $\mu$ and $\mu_\varnothing$ be defined as in \eqref{eq:mus} and $\eta(j) = Z_j^{\mu,\text{seq}}$, $j\geq 0$. It holds that
\begin{align*}
\nu_\eta = \begin{cases}
\infty & \text{if }~\nu_\mu \geq 1,\\
\infty & \text{if }~0 <\nu_\mu < 1,~ \rho_\mu \geq 1, \\
\frac{A_\mu(1-\nu_\mu)+\rho_{\mu}B_\mu}{(1-A_\mu)(1-\nu_\mu)} & \text{if }~0 <\nu_\mu < 1, ~\rho_\mu < 1
\end{cases}
\end{align*}
where
\begin{align*}
A_\mu = \rho_\mu\frac{g_\varnothing^\mu(\rho_\mu)}{g^\mu(\rho_\mu)} \qquad \text{and}\qquad B_\mu = \rho_\mu \frac{(g_\varnothing^\mu)'(\rho_\mu)}{g^\mu(\rho_\mu)}.
\end{align*}
\end{lemma}

\begin {proof}
By Lemma \ref{l:coupling} we may write
\begin{align}\label{eq:simplygenerated}
\mathcal{Z}^\eta(z) = zg^\eta(\mathcal{Z}^\eta(z))
\end{align}
with
\begin{align}\label{subcritical_scheme}
g^\eta(z) = \mathcal{Z}^{\mu, \text{seq}}(z) = \frac{1}{1-\mathcal{Z}^{\mu}_\varnothing(z)}.
\end{align} The second equality follows from a standard argument (see e.g.~\cite{stufler:2020}) which relates the generating function of the partition function for finite sequences of NCTs to the generating function of the partion function for NCTs. It follows that
\begin{align} \label{minroc}
\rho_\eta = \min\{R_\mu,y\}
\end{align} 
where $y$ is the unique solution to $\mathcal{Z}_\varnothing^\mu(y) = 1$.
Recall that 
\begin{align*}
\nu_\eta = \lim_{x\nearrow \rho_\eta} \mathbf{m}_\eta(x)
\end{align*}
where 
\begin{align} \label{them}
\mathbf{m}_\eta(x) = \frac{x (g^\eta)'(x)}{g^\eta(x)} = \frac{x (\mathcal{Z}^{\mu}_\varnothing)'(x)}{1-\mathcal{Z}^{\mu}_\varnothing(x)} = x (\mathcal{Z}^{\mu}_\varnothing)'(x) \sum_{n=0}^\infty (\mathcal{Z}^{\mu}_\varnothing(x))^n.
\end{align}
Differentiating \eqref{treeformula} and \eqref{treeformula2} we get
\begin{align*}
(\mathcal{Z}^{\mu})'(x) &= g^\mu(\mathcal{Z}^\mu(x))+x (g^\mu)'(\mathcal{Z}^\mu(x)) (\mathcal{Z}^\mu)'(x) \\
(\mathcal{Z}^{\mu}_\varnothing)'(x) &= g^\mu_\varnothing(\mathcal{Z}^\mu(x))+x (g_\varnothing^\mu)'(\mathcal{Z}^\mu(x)) (\mathcal{Z}^\mu)'(x)
\end{align*}
which solved together for $(\mathcal{Z}_\varnothing^\mu)'(x)$ yields
\begin{align}\label{Zprime}
(\mathcal{Z}^{\mu}_\varnothing)'(x) = g^\mu_\varnothing(\mathcal{Z}^\mu(x)) + \frac{x (g_\varnothing^\mu)'(\mathcal{Z}^\mu(x))g^\mu(\mathcal{Z}^\mu(x))}{1-x(g^\mu)'(\mathcal{Z}^\mu(x))}.
\end{align}

We first assume that $\nu_\mu \geq 1$ and show that $\nu_\eta = \infty$. By \eqref{them}, we may view $\mathbf{m}_{\eta}(x)$ as a power series in $x$ with positive coefficent whose radius of convergence equals 
\begin{align*}
r=\min\{R_\mu,y\} = \rho_\eta
\end{align*}
where $y$ is the unique value such that $\mathcal{Z}_\varnothing^\mu(y)=1$. It is therefore sufficient to show that $\mathbf{m}_{\eta}(\rho_\eta) = \infty$. Since $\nu_\mu \geq 1$ then by \eqref{taucondition} there is a $\tau_\mu \leq \rho_\mu = \rho_{\varnothing,\mu}$ such that 
$\mathbf{m}_\mu(\tau_\mu) = 1$ and we also recall the relations \eqref{tauandR} which state that $\tau_\mu = \mathcal{Z}^\mu(R_\mu)$ and $R_\mu = \tau_\mu/g^\mu(\tau_\mu)$. Thus,
\begin{align*}
(\mathcal{Z}^{\mu}_\varnothing)'(R_\mu) &= g^\mu_\varnothing(\tau_\mu) + \frac{R_\mu (g_\varnothing^\mu)'(\tau_\mu)g^\mu(\tau_\mu)}{1-R_\mu (g^\mu)'(\tau_\mu)} \\
&=  g^\mu_\varnothing(\tau_\mu) + \frac{\tau_\mu (g_\varnothing^\mu)'(\tau_\mu)}{1-\mathbf{m}_\mu(\tau_\mu)} = \infty.
\end{align*}
From this, \eqref{minroc} and  \eqref{them} it follows that $\mathbf{m}_\eta(\rho_\eta) = \infty$. 

Next, assume that $0 < \nu_\mu < 1$. Then by \eqref{taucondition} $\tau_\mu = \rho_\mu$. Firstly, if $\rho_\mu \geq 1$ then $\mathcal{Z}^\mu(R_\mu) =\tau_\mu \geq 1$ and therefore there is a $y\leq R_\mu$ such that $\mathcal{Z}^\mu(y) = 1$. Thus $\rho_\eta = y$ and we see immediately from \eqref{them} that $\mathbf{m}_\eta(\rho_\eta) = \infty$ and thus $\nu_\eta = \infty$. Finally, if $\rho_\mu < 1$ then $\rho_\eta = R_\mu$ and we get from \eqref{treeformula2}, \eqref{them} and \eqref{Zprime} that
\begin{align*}
\nu_\eta &= \mathbf{m}_\eta(\rho_\eta) = \mathbf{m}_\eta(R_\mu)  = \frac{R_\mu(\mathcal{Z}^{\mu}_\varnothing)'(R_\mu)}{1-\mathcal{Z}^{\mu}_\varnothing(R_\mu)} \\
&= \frac{\tau_\mu}{g^\mu(\tau_\mu)}\left[\frac{g_\varnothing^\mu(\tau_\mu)+\frac{\tau_\mu (g_\varnothing^\mu)'(\tau_\mu)}{1-\nu_\mu}}{1-\tau_\mu\frac{g_\varnothing^\mu(\tau_\mu)}{g^\mu(\tau_\mu)}}\right]
\end{align*}
The result follows by recalling that $\tau_\mu = \rho_\mu$ in the current case. 
\end{proof}
\begin{remark}
	If $\mu= \mu_\varnothing$ then $A_\mu= \rho_\mu$ and $B_\mu = \nu_\mu$ and we recover the formula in \cite{stufler:2020}.
\end{remark}

\section{Proofs of the limit theorems}\label{sec:result}
In this section we provide more details for the proofs of Theorems \ref{mainresultCRT1} - \ref{mainresult2}. In each case we show how existing techniques may be adapted to the current situation. In the first three subsection we introduce some important definitions. Firstly, we introduce the \emph{\L ukasiewicz path} of a tree which is a central object in the proofs, secondly we define the notion of a discrete looptree and finally we recall the definition of the Gromov-Hausdorff metric.

\subsection{\L ukasiewicz path}
Let $T$ be a plane tree whose vertices are listed in depth first-search-order(or lexicographical order), $\varnothing = v(0) < v(1) < ... < v(|T|-1)$. The \L ukasiewicz path $W(T) = (W_{j}(T), 0 \leq j \leq |T|)$ is defined by $W_{0}(T):= 0$ and $$W_{j}(T) := \sum_{i=0}^{j-1}(\out(v(i)) - 1)$$  for $1 \leq j \leq |T|)$. Figure \ref{Lukapath} gives an example of a tree encoded by a \L ukasiewicz path. $W$ is viewed as a c\`adl\`ag function (right continuous with left limits) by assuming it is constant in between successive integers. It may be seen that $W_{j}(T) \geq 0$ for $0\leq j \leq |T|-1$, but $W_{|T|}(T)=-1$.

\begin{figure}[ht]
 \centering
 \resizebox{0.9\textwidth}{!}{
 \includegraphics[width=25cm]{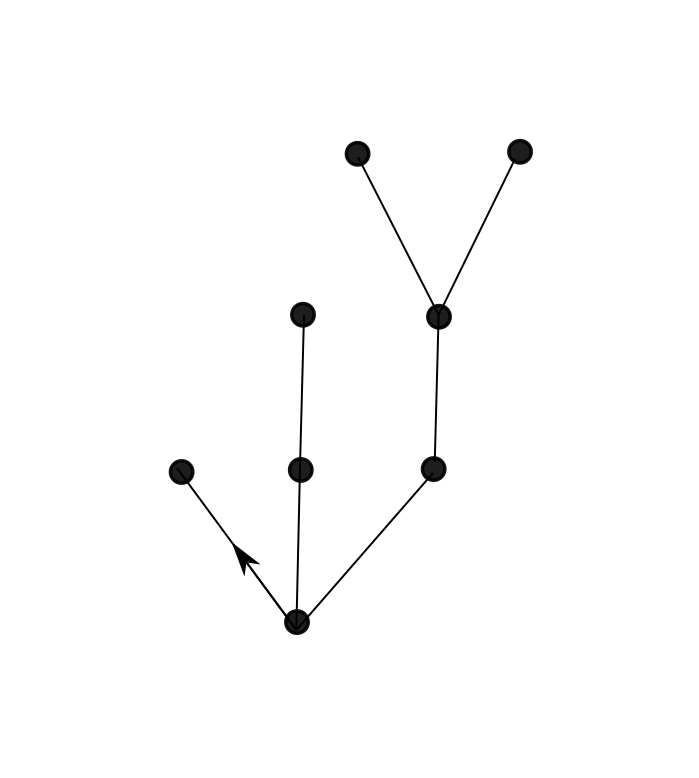}
		\hfill 
	\includegraphics[width=40cm]{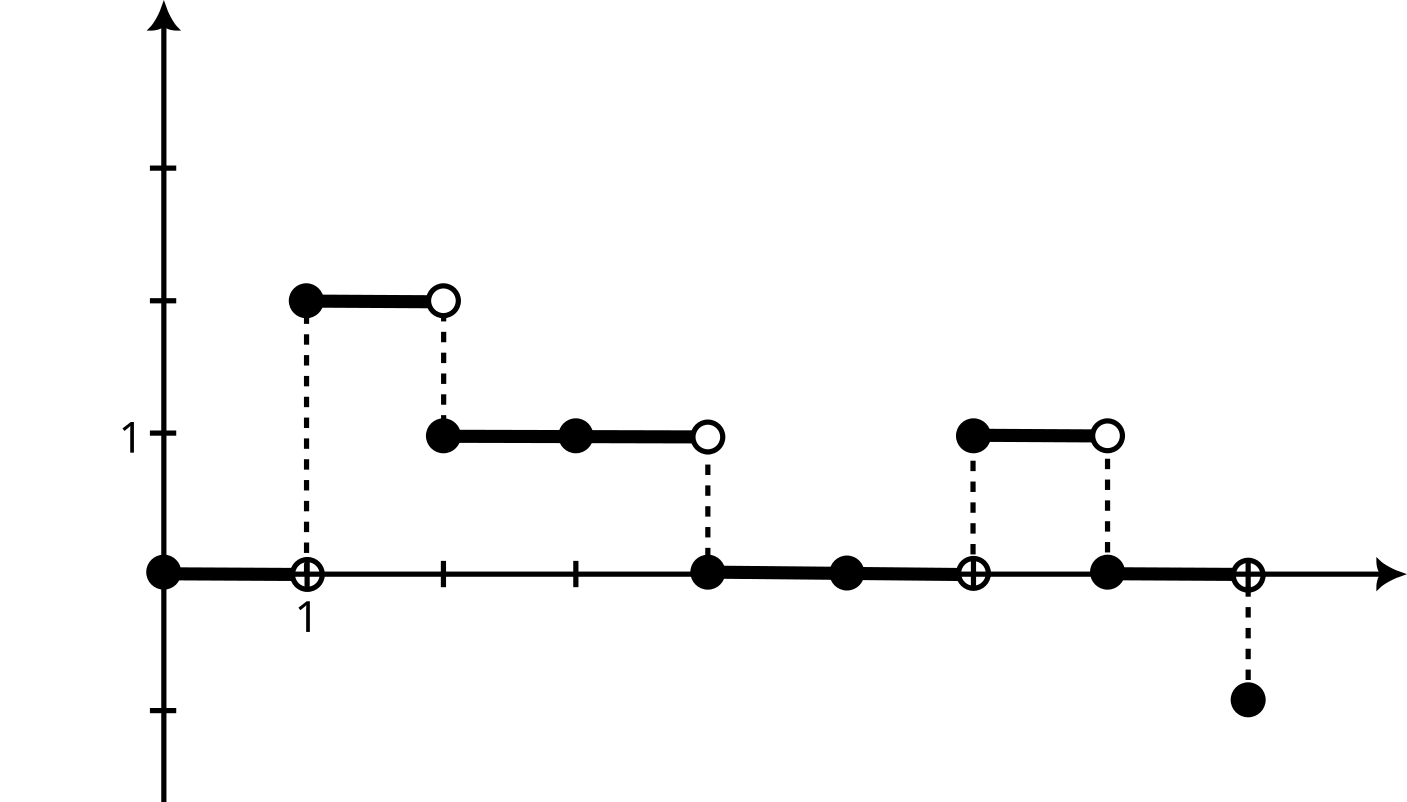}
	\hfill 
}
 
\caption{A tree encoded by a \L ukasiewicz path.}
\label{Lukapath}
\end{figure}     

Consider a sequence ($\tau_n$) distributed by $\GW^\mu_n$. If $\mu$ is in the domain of attraction of a stable law with index $\alpha\in(1,2)$, then 
\begin{align} \label{eq:lukconv}
	\left(\frac{1}{b_{n}}W_{\floor*{nt}}(\tau_{n}); 0 \leq t \leq 1 \right) \xrightarrow[n \rightarrow \infty]{\enskip (d) \enskip}X^{exc,\alpha},
\end{align}
where $X^{exc,\alpha}$ is an excursion of an $\alpha$-stable spectrally positive L\'evy process (see \cite{DuquesneGW,DuqGalllevypro}). Background information on L\'evy process can be found in \cite{BertoinLevy} and for more details, see \cite{CurienKortchloop}. This process lives on the Skorokhod space $D([0,1],\mathbb{R})$ of real valued c\`adl\`ag functions.

\subsection{Discrete looptrees}
A discrete looptree is constructed on the vertex set of a plane tree, $T$ as follows: For each vertex $v$ of outdegree $k \geq 1$, add the edges $\{v,v1\}$ and $\{v,vk\}$ (with repetition if $k=1$), and add the edges $\{vi,v(i+1)\}$, for $i = 1,\ldots, k-1$.
See Fig.~\ref{looptree} for an illustration. We denote the associated looptree by $\Loop(T)$. $\Loop(T)$ will be viewed as a metric space equipped with the graph distance as its metric (each edge with length one). 

\subsection{The Gromov-Hausdorff topology.} \label{ss:GH} Let $(E_{1},d_{1})$,$(E_{2},d_{2})$ be two compact metric spaces. The Gromov-Hausdorff distance between $E_{1}$ and $E_{2}$ is \\

\[
\mathrm{d}_{\mathrm{GH}}(E_{1},E_{2})=\inf_{\substack{(E,d)\,\mathrm{metric}\,\mathrm{space}\\\varphi_{1}:E_{1}\rightarrow E\\\varphi_{2}:E_{2}\rightarrow E}}\mathrm{d}_{\mathrm{Hau}}^{E}(\varphi_{1}(E_{1}),\varphi_{2}(E_{2}))
\]
where $\varphi_{1}$ and $\varphi_{2}$ are isometries. Recall that the Hausdorff distance between two compact subsets, $K$ and $K\textprime$ of the metric space $E$ is 
\[\mathrm{d}_{\mathrm{Hau}}^{E}(K,K\textprime) = \inf\{\varepsilon > 0~:~ K\subseteq K\textprime_{\varepsilon}\,\mathrm{and}\,K\textprime \subseteq K_{\varepsilon}\}\]
where $K_{\varepsilon}$ denotes the $\varepsilon$-neighborhood of $K$. An equivalent definition of the Gromov-Hausdorff distance involves correspondences. A correspondence between two metric spaces, $(E_{1},d_{1})$ and $(E_{2},d_{2})$ is a subset $\mathcal{R}$ of $E_{1} \times E_{2}$ such that for every $x_{1}\in E_{1}$, there exists at least one point $x_{2}\in E_{2}$ such that $(x_{1},x_{2})\in \mathcal{R}$ and conversely for every $y_{2}\in E_{2}$, there exists at least one point $y_{1}\in E_{1}$ such that $(y_{1},y_{2})\in \mathcal{R}$. The distortion of the correspondence $\mathcal{R}$ is defined by 
\[\mathrm{dis}(\mathcal{R}) = \sup\{\vert d_{1}(x_{1},y_{1}) - d_{2}(x_{2},y_{2}) \vert:(x_{1},x_{2}),(y_{1},y_{2})\in \mathcal{R}\}.\]
The Gromov-Hausdorff distance can be written in term of correspondences by the formula,
\[\mathrm{d}_{\mathrm{GH}}(E_{1},E_{2})=\frac{1}{2}\inf\{\mathrm{dis}(\mathcal{R})\}\]
where the infimum is over all correspondences $\mathcal{R}$ between $E_{1}$ and $E_{2}$. We refer the reader to \cite{metricgeo} for details about this metric.\\

Let $(X,d)$ be a compact metric and denote the diameter of $X$ by $\diam_{d}(X)$. We will often suppress $d$ from the notation if it is the intrinsic metric or graph metric on $X$. If $G$ is a graph we denote the vertex set of $G$ by $V_G$ (or sometimes by $G$) and the graph metric on $V_G$ by $d_G$. The following result will be used repeatedly and its proof is included for completeness. 
\begin{lemma} \label{l:GHestimate}
Let $G$ be a graph and $H$ and $\bar H$ be subgraphs of $G$ such that $G = H \cup \bar H$ and $H \cap \bar H = e$ is a single edge $e = \{v_1,v_2\}$ (with the possibility that $v_1 = v_2$). Then
\begin{align*}
	d_{\mathrm{GH}}((V_G,d_G),(V_H,d_H)) \leq \diam_{d_{\bar{H}}}(V_{\bar H}).
\end{align*}
\end{lemma}
\begin{proof}
The identity map on $V_G$ maps $(V_H,d_H)$ isometrically to $(V_G,d_G)$ since $H\cap \bar H$ is a single edge.
Therefore
\begin{align*}
	d_{\mathrm{GH}}((V_G,d_G),(V_H,d_H)) &\leq d_{\mathrm{Hau}}^G(V_G,V_H) = \inf\{\varepsilon\geq 0 ~:~V_G \subseteq (V_H)_\varepsilon \} \\
	&= \max_{v\in\{v_1,v_2\}} \max_{w\in V_{\bar H}} d_{\bar H}(v,w) \leq \diam_{\bar H}(V_{\bar H}).
\end{align*}
\end{proof}

\bigskip
\subsection{Proof of Theorem \ref{mainresultCRT1}}\label{ss:proof1}
We rely on the Markov-chain argument established in \cite{randissection}. This argument involves introducing many new concepts and instead of repeating all their definitions here, we aim to point out how our case is different and in which way that modifies the result. 

Let $\mathcal{T}^w_n = \phi(\mathcal{H}_n^{w,\star})$ be the weak dual of $\mathcal{H}_n^{w,\star}$ and choose $\mu$ and $\mu_\varnothing$ as in \eqref{eq:mus} such that they are probability measures satisfying the assumptions of Theorem \ref{mainresultCRT1}. By Theorem \ref{rem:G-W}, the shape $\mathcal{S}(\mathcal{T}_n^{w})$ is distributed as $\GW^{\mu,\mu_\varnothing}_n$.  
Recall that $\mathcal{T}_n^{w,\bullet}$ denotes the augmented tree obtained by attaching leaves by an edge to each corner of $\mathcal{T}_n^w$. In \cite{randissection}, the setup is similar, they consider a random face weighted dissection $\mathcal{D}^\mu_n$, (not necessarily non-crossing), whose augmented weak dual $\hat{\mathcal{T}}_n^{\mu,\bullet} := \phi(\mathcal{D}^\mu_n)^\bullet$ is distributed as a planted $\GW^\mu$ tree conditioned on having $n$ leaves. The difference to our case is thus only that we consider an augmented, modified $(\mu,\mu_\varnothing)$-Galton-Watson tree instead of a planted normal $\mu$-Galton-Watson tree and we condition on the number of vertices instead of the number of leaves.

 Kortchemski and Marzouk \cite[Theorem 21 and Remark 26]{Noncrossing}  showed that
\begin{align*}
n^{-1/2} \cdot \mathcal{T}^w_n \xrightarrow[n \rightarrow \infty]{\enskip (d) \enskip} c_{\mathrm{tree}}(\mu)\cdot \mathcal{T}_{\mathbf{e}}
\end{align*}
with $c_{\mathrm{tree}}(\mu) := 2\sigma_\mu^{-1}$. This is the same as Aldous' well-known result but for the modified Galton-Watson tree $\GW^{\mu,\mu_\varnothing}_n$. Since
\begin{align*}
	d_{\mathrm{GH}}(\mathcal{T}_n^{w},\mathcal{T}_n^{w,\bullet})\leq 1
\end{align*}
one obtains similarly that
\begin{align} \label{eq:scaling_limit_modified}
	n^{-1/2} \cdot \mathcal{T}^{w,\bullet}_n \xrightarrow[n \rightarrow \infty]{\enskip (d) \enskip} c_{\mathrm{tree}}(\mu)\cdot \mathcal{T}_{\mathbf{e}}.
\end{align}
We note that the constant $c_{\mathrm{tree}}(\mu)$ is different from the one in \cite{randissection} which is only due to the fact that they condition on the number of leaves. 

The key argument in \cite{randissection} involves showing how geodesics in a dissection $\mathcal{D}$ may be found from geodesics in the corresponding augmented weak dual tree $\hat{\mathcal{T}}^\bullet = \phi(\mathcal{D})^\bullet$ by an iterative local construction. This comparison of geodesics works deterministically but is then applied to the random objects. The first comparison is made on the local limits of the two objects so we start by describing the local limit in the current case.

First, define the size-biased distributions 
\begin{align*}
	\mu^\ast(k) = k\mu(k), \qquad k\geq 1
\end{align*}
and
\begin{align*}
	\mu_\varnothing^\ast(k) = \frac{k\mu_\varnothing(k)}{\sum_{j\geq 1} j\mu_\varnothing(j)}, \qquad k\geq 1.
\end{align*}
Then define an infinite tree $\mathcal{T}_\infty^w$ as folllows. Let $(W_1,W_2,\ldots)$ be an infinite path, often referred to as a \emph{spine}. The vertex $W_1$ has offspring according to $\mu^\ast_\varnothing$, one of which is chosen uniformly at random to be $W_1$. Each $W_j$, $j\geq 1$ has independently offspring according to $\mu^\ast$, one of which is chosen uniformly at random to be $W_{j+1}$. From each of the offspring, excluding $W_1,W_2,\ldots$, we grow independent $\GW^\mu$ trees, see Fig.~\ref{fig:locallimit}. 

   \begin{figure}[ht]
	\centering
	\resizebox{1.\textwidth}{!}{
		\includegraphics[width=8.9cm]{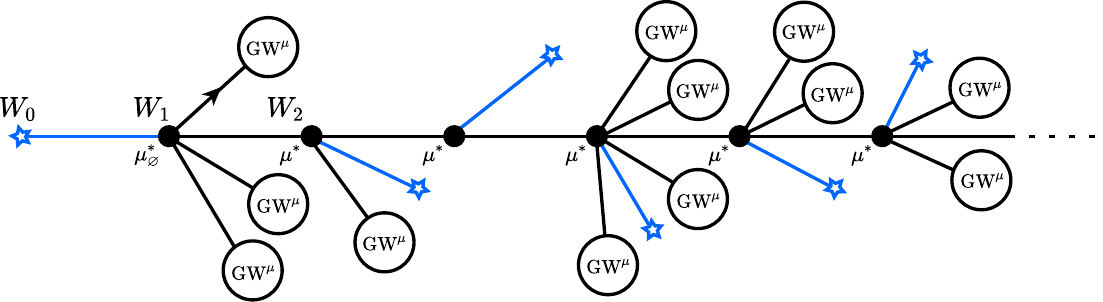}
	}  
	\caption{The tree $\mathcal{T}^w_\infty$. Its augmentation $\mathcal{T}^{w,\bullet}_\infty$ is depicted by blue edges and star-shaped vertices.  \label{fig:locallimit}}
\end{figure}

The augmented infinite tree $\mathcal{T}_\infty^{w,\bullet}$ is obtained by first sampling $\mathcal{T}_\infty^w$ and then independently for each vertex different from $W_1$, a corner is sampled uniformly and a new leaf attached to it by an edge. A leaf, denoted by $W_0$, is also attached to $W_1$ at the corner which is counter-clockwise from the root-edge. This last step is what is referred to as \emph{planting} in \cite{randissection}. 

Let $k\geq 0$ be an integer and denote by $[T]_k$ the subtree of $T$ consisting of its first $k$ generations. By a minor adjustment of standard arguments concerning local limits (see e.g.~\cite[Theorem 7.1]{janson}), it holds that 
\begin{align*}
	[\mathcal{T}_n^{w,\bullet}]_k \xrightarrow[n \rightarrow \infty]{\enskip (d) \enskip} [\mathcal{T}_\infty^{w,\bullet}]_k.
\end{align*}
There is a corresponding infinite dual dissection $\mathcal{H}_\infty^{w,\star}$, and we denote the edges in that dissection which are dual to $(W_i,W_{i+1})$ by $e_i = (e_{i}^L,e_{i}^R)$, i $\geq 0$ where $e_i^L$ (resp.~$e_i^R$) denotes the vertex to the left (resp.~right) of the spine.

Consider the random sequence $(S_n)_{n\geq 0}$ defined by
\begin{align*}
S_n = \min\{d^\star(e_0^L,e_n^{L}),d^\star(e_0^L,e_n^{R})\}
\end{align*} 
where $d^\star$ denotes the graph distance in $\mathcal{H}_\infty^{w,\star}$. The process $(S_n)_{n\geq 0}$ provides information on distances between points in the dissection, as a function of the number of steps $n$ taken along the spine in the tree. In order to find the value $S_n$, a geodesic whose length is $S_n$ is constructed step by step. This geodesic has to go through either $e_i^L$ or $e_i^R$ for each $i\geq 0$. Assuming we have constructed the first $i$ steps, in order to construct the next one we need to know whether we went through $e_i^L$ or $e_{i}^R$ or whether it is still undecided. The information on this position is stored in a variable $P_i$ (position) which takes one of the three values: $L$ (left), $R$ (right) or $U$ (undecided).  We refer to \cite[Section 2.2]{randissection} for the precise description. 

Let $X_0 = 0$, $X_n = S_{n}-S_{n-1}$, $n\geq 1$ denote the increments of $S_n$. It is shown in \cite{randissection} that $(X_n,P_n)_{n\geq 0}$ is a Markov-chain whose transition probabilities only depend on the value of $P_n$. The sequence $(S_n,P_n)_{n\geq 0}$ is called a Markov additive process with driving chain $(P_n)_{n\geq}$. Step $i$ of the chain can be found by comparing the number of children of $W_i$ to the left and to the right of the spine, including the additional leaf added in the augmentation. Denote the number of children of $W_{n+1}$ in $\mathcal{T}_\infty^{w,\bullet}$  which lie to the left (resp.~right) of the spine by by $G_n$ (resp.~$D_n$), and the total number of its children by $C_{n+1} = G_n + D_n + 1$. 

By comparing with the transition probabilities in \cite[pages 11-12]{randissection} we find that
\begin{align*}
	\mathbb{P}(X_{n+1} = i, P_{n+1} = R~|~P_n = R) &=  \mathbb{P}(G_{n} \geq i, D_{n} = i)\\
	\mathbb{P}(X_{n+1} = i, P_{n+1} = L~|~P_n = R) &=  \mathbb{P}(G_{n} \geq i+1, D_{n}= i - 1) \\
	\mathbb{P}(X_{n+1} = i, P_{n+1} = U~|~P_n = R) &=  \mathbb{P}(G_{n} = i, D_{n}= i -1) \\
		\mathbb{P}(X_{n+1} = i, P_{n+1} = R~|~P_n = U) &= \mathbb{P}(G_{n} \geq i+1, D_{n}= i)  \\
			\mathbb{P}(X_{n+1} = i, P_{n+1} = U~|~P_n = U) &=  \mathbb{P}(G_{n} =i, D_{n}= i). \\
\end{align*}
By symmetry between left and right, the first four lines also give the transition probabilities when $L$ and $R$ switch roles.

Recall that $C_{n+1} = G_n + D_n+1$, and that for $n\geq 1$
\begin{align*}
\mathbb{P}(C_{n+1}-1 = j) = \mu^\ast(j) = j\mu(j)
\end{align*}
and by symmetry between left and right
\begin{align*}
\mathbb{P}(G_n = i~|~ C_{n+1} = j) = \mathbb{P}(D_n = i~|~ C_{n+1} = j) = \frac{1}{j}\mathbbm{1}\{0\leq i \leq j-1\}.
\end{align*}
From this we can find formulas for the transition probabilities in terms of $\mu$. For easier notation, define $\gamma(j) = j\mu(j)/(j+1)$ and for any $A\subseteq \mathbb{Z}_+ = \{0,1,\ldots\}$, and any $i \in \mathbb{Z}_+$ let
\begin{align*}
	\gamma(A) = \sum_{j\in A} \gamma(j) \qquad \text{and}\qquad \bar\gamma(i) = \sum_{j\geq i}\gamma(j).
\end{align*}
Then we find for example that 
\begin{align*}
\mathbb{P}(G_n\geq i, D_n = i) &= \sum_{j\geq 1} \mathbb{P}(G_n\geq i, D_n = i~|~ C_{n+1}-1=j)\mathbb{P}(C_{n+1}-1=j) \\
&= \sum_{j\geq 2i} \mathbb{P}(D_n = i~|~C_{n+1} = j+1) \mu^\ast(j) \\ &= \sum_{j\geq 2i} \frac{j}{j+1}\mu(j) = \bar{\gamma}(2i)
\end{align*}
and with similar calculations we get
\begin{align*}
\mathbb{P}(G_{n} \geq i+1, D_{n}= i - 1) &=  \bar{\gamma}(2i)\mathbbm{1}\{i\geq 1\},\\
 \mathbb{P}(G_{n} = i, D_{n}= i -1) &= \gamma(2i-1)\mathbbm{1}\{i\geq 1\}, \\
  \mathbb{P}(G_{n} \geq i+1, D_{n}= i) &= \bar \gamma(2i+1), \\
  \mathbb{P}(G_{n} =i, D_{n}= i)  &= \gamma(2i).
\end{align*}

Note that these formulas only hold for $n\geq 1$ since the degree distribution of $W_1$ is different from the others. However,  the first step $X_1$ does not affect the asymptotic behaviour of $S_n$ and therefore we leave out the details.

Because of symmetry between left and right we proceed as in \cite{randissection} and say that a state $P_n$ is determined, denoted by $P_n = D$ if and only if $P_n \in \{L,R\}$. Then
\begin{align*}
	\mathbb{P}(X_{n+1} = i, P_{n+1} = D~|~P_n = D) &=  \bar\gamma(2i) + \bar\gamma(2i)\mathbbm{1}\{i \geq 1\}\\
	\mathbb{P}(X_{n+1} = i, P_{n+1} = U~|~P_n = D) &= \gamma(2i-1)\mathbbm{1}\{i\geq 1\}  \\
	\mathbb{P}(X_{n+1} = i, P_{n+1} = D~|~P_n = U) &= 2\bar\gamma (2i+1) \\
	\mathbb{P}(X_{n+1} = i, P_{n+1} = U~|~P_n = U) &= \gamma (2i).
\end{align*}

By summing over $i$ in the above formulas we find the transition probabilites of the driving chain $(P_n)_{n\geq 0}$ to be

\begin{align*}
	\mathbb{P}(P_{n+1} = U~|~P_n = D) &=  \gamma(2\mathbb{Z}_+ +1) = 1 - 	\mathbb{P}(P_{n+1} = D~|~P_n = D) \\
			\mathbb{P}(P_{n+1} = U~|~P_n = U) &=  \gamma(2\mathbb{Z}_+) = 1 - 	\mathbb{P}(P_{n+1} = D~|~P_n = U). \\
\end{align*}
This allows us to find the stationary state of the driving chain
\begin{align*}
\pi(D) = \frac{1-\gamma(2\mathbb{Z}_+)}{1-\gamma(2\mathbb{Z}_+)+\gamma(2\mathbb{Z}_+ + 1)}, \qquad \pi(U) = \frac{\gamma(2\mathbb{Z}_+ + 1)}{1-\gamma(2\mathbb{Z}_+)+\gamma(2\mathbb{Z}_+ + 1)}
\end{align*}
as long as $\gamma(2\mathbb{Z}_+ +1) > 0$. 
 The constant $c_{\mathrm{geo}}(\mu)$ appearing in Theorem \ref{mainresultCRT1} may now be calculated by finding the expected value of the step distribution $X_n$, for $n\geq 2$, when the driving chain is in its stationary state. Using this, along with \eqref{eq:scaling_limit_modified},  allows us to complete the proof. However, we  require a few additional technical details which are the same, mutatis mutandis, as in \cite[pages 12-18]{randissection} so we leave out the details. We therefore conclude by calculating
\begin{align*}
	c_\mathrm{geo}(\mu) &= \mathbb{E}_\pi(X_2) = \sum_{i\geq 0} i\mathbb{P}(X_2 = i~|~P_1 = D)\pi(D) + i\mathbb{P}(X_2 = i~|~P_1 = U)\pi(U) \\
	&= \frac{1}{1-\gamma(2\mathbb{Z}_+)+\gamma(2\mathbb{Z}_+ + 1)}\Bigg(\Big(\sum_{i\geq 1 }i\bar{\gamma}(2i)+\sum_{i\geq 1}i\bar\gamma (2i-1)\Big)(1-\gamma(2\mathbb{Z}_+)) \\
	&+\Big(\sum_{i\geq 0}i \bar \gamma(2i+1)+\sum_{i\geq 0} i \bar\gamma(2i)\Big)\gamma(2\mathbb{Z}_++1)\Bigg).
\end{align*}
By swapping the sum over $i$ and the sum appearing in $\bar \gamma$ in each term we find that
\begin{align*}
\sum_{i\geq 1}i\bar \gamma(2i)+ \sum_{i\geq 1}i\bar \gamma(2i-1)  &= \frac{1}{2}\sum_{j\geq 1} \gamma(j)\left(\left\lfloor \frac{j}{2} \right\rfloor \left(\left\lfloor\frac{j}{2} \right\rfloor +1\right)+\left\lfloor \frac{j+1}{2} \right\rfloor \left(\left\lfloor\frac{j+1}{2} \right\rfloor +1\right)\right) \\
&= \frac{1}{2} \sum_{j\geq 1} \gamma(j)\left(\frac{1}{2}j(j+1)+ \frac{1}{2}\mathbbm{1}\{j~\mathrm{odd}\}\right) \\
&= \frac{1}{4} \sum_{j\geq 1} j^2 \mu(j) + \frac{1}{4} \gamma(2\mathbb{Z}_+ + 1) \\
&= \frac{\sigma_\mu^2 +1 + \gamma(2\mathbb{Z}_+ + 1)}{4}
\end{align*}
and similarly
\begin{align*}
\Big(\sum_{i\geq 0}i \bar \gamma(2i+1)+\sum_{i\geq 0} i \bar\gamma(2i)\Big) &=  \frac{\sigma_\mu^2 +1 - \gamma(\mathbb{Z}_+)- \gamma(2\mathbb{Z}_+ + 1)}{4}
\end{align*}
which finally yields
\begin{align*}
c_\mathrm{geo}(\mu) &= \frac{1}{4}\left(\sigma_\mu^2 + 1 + \frac{\gamma(2\mathbb{Z}_++1)(1-2\gamma(\mathbb{Z}))}{1-\gamma(2\mathbb{Z}_+)+\gamma(2\mathbb{Z}_+ +1)}\right).
\end{align*}

\begin{flushright} $\square$ \end{flushright}

\subsection{Proof of Theorem \ref{mainresult1}} \label{ss:proof2}
The following invariance theorem gives a sufficient condition on a sequence of random plane trees $(\tau_{n})_{n\geq 1}$ that ensures that the associated looptrees, $(\Loop(\tau_{n}))_{n \geq 1}$ appropriately rescaled, converge towards the $\alpha$-stable looptree $\mathcal{L}_{\alpha}$. Informally, one can view the $\alpha$-stable looptree $\mathcal{L}_{\alpha}$ as "$\Loop(\mathcal{T}_\alpha)$", where $\mathcal{T}_\alpha$ is the $\alpha$-stable tree mentioned in the introduction. The precise definition requires more work and we refer to \cite[Sec. 2]{CurienKortchloop} for the details of the construction.

\begin{theorem}[Curien \& Kortchemski, \cite{CurienKortchloop}]\label{invariance}
	Let $(\tau_{n})_{n\geq 1}$ be a sequence of random trees such that there exists a sequence $(b_{n})_{n\geq 0}$ of positive real numbers satisfying
	\begin{enumerate}
		\item[(i)] $\left(\frac{1}{b_{n}}W_{\floor*{|\tau_{n}|t}}(\tau_{n}); 0 \leq t \leq 1 \right) \xrightarrow[n \rightarrow \infty]{\enskip (d) \enskip}X^{exc,\alpha}$,
		\item[(ii)]$ \frac{1}{b_{n}}\Height(\tau_{n}) \xrightarrow[n \rightarrow \infty]{\enskip (\mathbb{P}) \enskip}0$
	\end{enumerate}
	where the first convergence holds in distribution for the Skorokhod topology on $D([0,1],\mathbb{R})$ and the second convergence holds in probability. Then the convergence $$\frac{1}{b_{n}}\cdot \textsf{Loop}(\tau_{n})\xrightarrow[n \rightarrow \infty]{\enskip (d) \enskip}\mathcal{L}_{\alpha} $$ holds in distribution for the Gromov-Hausdorff topology.
\end{theorem}

Let $\mu$ and $\mu_\varnothing$ be given by \eqref{eq:mus} and chosen such that they are probability measures satisfying the assumptions of Theorem \ref{mainresult1}. First of all, we apply this invariance principle to   $\mathcal{T}^w_n = \mathcal{S}(\phi_n(\mathcal{H}_n^{w,\star}))$, the shape of the weak dual of $\mathcal{H}_n^{w,\star}$ which by Proposition \ref{rem:G-W} is distributed as $\GW^{\mu,\mu_\varnothing}_n$. Condition (i) is satisfied for $\mathcal{T}^w_n$ by \cite[Theorem 21]{Noncrossing} whose proof involves comparing $\GW^{\mu,\mu_\varnothing}_n$ with $\GW^{\mu}_n$ and using that (i) holds for the latter by standard results \cite{DuquesneGW}. Property (ii) holds since the height of $\mathcal{T}^w_n$ is of order $n/b_n$ and that $b_n$ is regularly varying with index $\alpha < 2$. The former statement follows again from comparing $\GW^{\mu,\mu_\varnothing}_n$ with $\GW^{\mu}_n$, see the discussion in \cite[Remark 26]{Noncrossing}. We therefore have
\begin{align*}
	\frac{1}{b_n} \cdot \mathsf{Loop}(\mathcal{T}^w_n) \xrightarrow[n \rightarrow \infty]{\enskip (d) \enskip} \mathcal{L}_{\alpha}. 
\end{align*}
Finally, we can bound the Gromov--Hausdorff distance between $\mathsf{Loop}(\mathcal{T}_n^w)$ and $\mathcal{H}^{w,\star}_n$. By \cite[Equation (4.10)]{CurienKortchloop} it holds deterministically (by a short argument) that the Gromov--Hausdorff distance between a dissection and the looptree of its weak dual tree is roughly bounded by its height. More precisely we have
\begin{align*}
\frac{1}{b_n}d_{\mathrm{GH}}(\mathcal{H}^{w,\star}_n,\mathsf{Loop}(\mathcal{T}_n^w)) \leq \frac{1}{b_n}\left(\Height(\mathcal{T}^w_n) + 2\right).
\end{align*}
The expression on the right hand side converges to 0 in probability since $\mathcal{T}^w_n$ satisfies (ii) in Theorem \ref{invariance}, which concludes the proof.

\begin{flushright} $\square$ \end{flushright}

\subsection{Proof of Theorem \ref{mainresultCRT2}} \label{ss:proof3}
Recall that the random NCO can in this case be sampled as described in Lemma \ref{l:coupling}. An underlying tree is first sampled according to branching weights $\eta$ given by $\eta(j) = Z_j^{\mu,seq}$. Each vertex of the underlying tree is then replaced by an i.i.d.~ copy of a finite sequence of $w$-face weighted dissections of the appropriate size. If the conditions in Theorem \ref{mainresultCRT2} are satisfied then, by Lemma \ref{nuetaconditions}, $\nu_\eta > 1$ which means that we are in case I) in \eqref{criteria}. This means that the branching weights $\eta$ of the underlying tree have a finite exponential moment, a condition which is referred to as Ia) in \cite{stufler:2020}. In addition, some very mild conditions on the decorations \cite[Conditions 1,2 and 3, page 83]{stufler:2020} are required which is easy to verify in the current case. We may then apply \cite[Theorem 6.60]{stufler:2020} which completes the proof. 

\begin{flushright} $\square$ \end{flushright}

\subsection{Proof of Theorem \ref{mainresult2}} \label{ss:proof4}
As in the previous subsection we rely on the description of the NCO as a decorated tree. In order to prove the theorem we apply the recent general invariance principle for decorated stable trees given in \cite[Theorem 5.1]{decoratedtrees}. In order to apply this result we need to show that Conditions D1, D2, T1, T2 and B1 in \cite[Section 3.1]{decoratedtrees} are satisfied (the other conditions stated there are only needed when measures are also considered). Since their statements are lengthy and require the introduction of more notation we will let it suffice to refer to them here. In the current case it trivially holds that
\begin{align} \label{eq:diamdiscrete}
\diam(\mathcal{SD}^w_n) \leq n/2.
\end{align}
Denote the circle of unit circumference by $C^1$. Then by definition
\begin{align} \label{eq:diamcont}
\diam(C^1) = 1/2.
\end{align}	
It will be sufficient to prove the two lemmas
\begin{lemma}\label{item1}
	The weights $\eta$ of the underlying tree are equivalent, in the sense of \eqref{eq:tilt1}, to a probability measure $\hat \eta$ which is critical and belongs to the domain of attraction of an $\alpha$-stable law, $\alpha \in (1,2)$. In particular, 
	\begin{align*}
	 	\hat \eta (j) \sim \frac{g^\mu(r) B_\mu}{1-A_\mu} L(j)  ((1-\nu_\mu)j)^{-\alpha-1}
	\end{align*}
	where $L$ is the same slowly varying function as in \eqref{muweights}.
\end{lemma}
\begin{lemma}\label{item2}
The following convergence holds in distribution in the Gromov-Hausdorff sense
\begin{align*}
n^{-1}\cdot \mathcal{SD}^w_n \to (1-\nu_\mu) \cdot C^1.
\end{align*}
\end{lemma}
Lemma \ref{item1} implies that T1 and T2 hold and along with \eqref{eq:diamdiscrete} it further implies that condition B1 holds (see \cite[Proposition 3.2 with $\gamma=1$ and Section 3.3]{decoratedtrees}). Lemma \ref{item2} implies that D1 holds and along with \eqref{eq:diamcont} it further implies that D2 holds.

These Lemmas give a rough idea for the validity of Theorem \ref{mainresult2}. By Lemma \ref{item1}, it holds that 
\begin{align*}
\hat \eta (j,\infty) \sim \frac{g^\mu(r) B_\mu}{\alpha(1-A_\mu)} L(j)  ((1-\nu_\mu)j)^{-\alpha-1}
\end{align*}
and therefore $\hat \eta$ belongs to the domain of attraction of a stable law with index $\alpha$. Thus, the finite underlying tree will converge towards an $\alpha$-stable tree once rescaled by
\begin{align*}
C_n &= (\alpha|\Gamma(-\alpha)|)^{1/\alpha}\inf\{x\geq 0~:~\hat\eta(x,\infty)\leq 1/n\} \\ &\sim  \left(\frac{|\Gamma(-\alpha)|g^\mu(r) B_\mu}{(1-\nu_\mu)(1-A_\mu)} \right)^{1/\alpha} (1-\nu_\mu)^{-1} (L(n)n)^{1/\alpha},
\end{align*}
see e.g.~\cite[Proof of Lemma 5.1]{stefansson:2019} for similar calculations.
By Lemma \ref{item2}, a typical large decoration converges towards a circle of circumference $1-\nu_\mu$. Hence, up to some technical details which are taken care of in the other conditions, the decorated tree converges towards a multiple $(1-\nu_\mu)$ of an $\alpha$-stable looptree. Combining these results gives the correct scaling of distances in the decorated trees in the theorem
\begin{align*}
	c_n = (1-\nu_\mu)C_n.
\end{align*}

\begin{proof}[Proof of Lemma \ref{item1}]
Define the probability measure 
\begin{align*}
	\hat \eta(j) = \frac{Z_j^{\mu,seq} R_\mu^j}{\mathcal{Z}^{\mu,seq}(R_\mu)}
\end{align*}
which is a tilt of $\eta(j) = Z_j^{\mu,seq}$.
The conditions of the Theorem state that $\mu$ satisfies $\nu_\mu < 1$ and
	\begin{align*}
	\mu(k) = L(k) k^{-\alpha-1} r^{-k}
\end{align*}
with $\alpha \in (1,2)$, $r=\rho_\mu$ and $L$ slowly varying, and that
	\begin{align*}
	\frac{A_\mu(1-\nu_\mu)+\rho_\mu B_\mu}{(1-A_\mu)(1-\nu_\mu)} = 1.
\end{align*}
In this case it holds that
\begin{align}
\mathcal Z^\mu(R_\mu) = r = \rho_\mu < 1 \qquad \text{and}\qquad \mathcal Z_\varnothing^\mu(R_\mu) = r \frac{g_\varnothing^\mu(r)}{g^\mu(r)} = A_\mu < 1.
\end{align}
In the following, we will use the notation $[z^j]\left\{ \sum_{j=0}^\infty a_n z^n \right\} = a_j$. Recall the expression \eqref{cond_partition} for $Z_n^\mu = [z^n]\{\mathcal{Z}^\mu(z)\}$. By \eqref{treeformula2} we have
\begin{align*} 
[z^j]\left\{\mathcal{Z}^\mu_\varnothing(z)\right\} &= 	  [z^{j-1}]\left\{g_\varnothing^\mu (\mathcal{Z}^\mu(z))\right\} \\
&\sim (g_\varnothing^\mu)'(\rho_\mu)  [z^{j-1}]\left\{\mathcal{Z}^\mu(z)\right\}   = R_\mu^{-1} B_\mu  [z^{j-1}]\left\{\mathcal{Z}^\mu(z)\right\}
\end{align*}
where we used \cite[Theorem 4.30]{foss} in the second last step. In the same way, we find that

\begin{align}
	\eta(j) = Z_j^{\mu,seq} &= [z^j]\{\mathcal{Z}^{\mu,seq}(z)\} = [z^j]\left\{\frac{1}{1-\mathcal{Z}^\mu_\varnothing(z)}\right\} \nonumber\\
	&\sim \frac{1}{(1-A_\mu)^2} [z^j] \{\mathcal{Z}_\varnothing^\mu(z)\} \nonumber\\
	&\sim  \frac{B_\mu}{R_\mu(1-A_\mu)^2}[z^{j-1}]\left\{\mathcal{Z}^\mu(z)\right\} \nonumber\\
	&\sim  \frac{r B_\mu}{(1-A_\mu)^2} L(j) R_\mu^{-j-1} ((1-\nu_\mu)j)^{-\alpha-1} \label{etaasym}
\end{align}
where in the last step we used $[z^{j-1}]\mathcal{Z}^\mu(z) = Z_{j-1}^\mu$ and \eqref{cond_partition}.
Then
\begin{align*}
	\hat \eta (j) \sim  \frac{r B_\mu}{R_\mu (1-A_\mu)} L(j)  ((1-\nu_\mu)j)^{-\alpha-1} = \frac{g^\mu(r) B_\mu}{(1-A_\mu)} L(j)  ((1-\nu_\mu)j)^{-\alpha-1}
\end{align*}
which shows that $\hat \eta$ is in the domain of attraction of a stable law with index $\alpha$. Finally, we calculate using \eqref{Zprime},
\begin{align*}
\sum_{j=0}^\infty j \hat \eta(j) &=  \frac{R_\mu(\mathcal{Z}^{\mu,seq})'(R_\mu) }{\mathcal{Z}^{\mu,seq}(R_\mu)} = \frac{R_\mu}{1-A_\mu}\left( g^\mu_\varnothing(r) + \frac{r (g_\varnothing^\mu)'(r)}{1-\nu_\mu}\right) \\
&= \frac{ A_\mu (1-\nu_\mu) + r B_\mu}{(1-A_\mu)(1-\nu_\mu)} =1
\end{align*}
which shows that $\hat \eta$ is critical.
\end{proof}

\begin{proof}[Proof of Lemma \ref{item2}] Since $\nu_\mu < 1$, the decorations, consisting of a finite sequence of NCDs, are in a so-called condensation phase, see e.g.~\cite{janson,jonsson, kortchemski}. This means that exactly one face, call it $f_{\max}$, will have a size close to $(1-\nu_\mu)n$  where $n$ is the total size of the NCDs in the sequence, and all other faces are uniformly small. The geometry of a large NCD may then be shown to be close to that of a circle, see Fig.~\ref{f:condensation}. We now make this more precise.  

 \begin{figure}[ht]
	\centering
	\resizebox{1.\textwidth}{!}{
		\includegraphics[width=0.95cm]{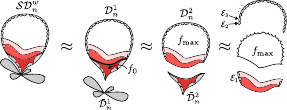}
	}  
	\caption{The figure demonstrates how one may show in three steps that $\mathcal{SD}_n^w$ is close to the circle graph $f_{\max}$ of size approximately $(1-\nu_\mu)n$, in the Gromov-Hausdorff sense.  The shaded regions represent small dissections. \label{f:condensation}}
\end{figure}

Denote the largest element of the finite sequence $\mathcal{SD}_n^w$	by $\mathcal{D}^{1}_n$ (in case of a tie, let it be the first largest one) and let $\bar{\mathcal{D}}^{1}_n$ denote the remaining part. Denote the root face of $\mathcal{D}^1_n$ by $f_0$. One may decompose $\mathcal{D}^1_n$ into the face $f_0$ and a  number of sub-dissections which have an edge in common with $f_0$. Denote the largest such sub-dissection by $\mathcal{D}^2_n$ (in case of a tie, choose e.g.~the first one in a clockwise order) and denote the remaining part by $\bar{\mathcal{D}}^{2}_n$. Denote the face in $\mathcal{D}_n^2$ of largest degree by $f_{\max}$. By abuse of notation we also let $f_{\max}$ denote the circle graph consisting of the edges on the boundary of $f_{\max}$. We will show that $\mathcal{SD}_n^w$ is close to $f_{\max}$ by going through the two intermediate dissections $\mathcal{D}^1_n$ and $\mathcal{D}^2_n$, see Fig.~\ref{f:condensation}. Since $\deg(f_{\max})$ is with high probability of size close to $(1-\nu_\mu)n$ this will conclude the proof. More precisely we will show that
\begin{align}
n^{-1}\cdot d_{\mathrm{GH}}(\mathcal{SD}_n^w, \mathcal{D}^1_n)  &\xrightarrow[n \rightarrow \infty]{\enskip (\mathbb{P}) \enskip} 0, \label{conv1}\\
n^{-1}\cdot d_{\mathrm{GH}}(\mathcal{D}^1_n,\mathcal{D}^2_n) &\xrightarrow[n \rightarrow \infty]{\enskip (\mathbb{P}) \enskip} 0, \label{conv2} \\
n^{-1}\cdot d_{\mathrm{GH}}(\mathcal{D}^2_n, f_{\max})  &\xrightarrow[n \rightarrow \infty]{\enskip (\mathbb{P}) \enskip} 0 \qquad \text{and} \label{conv3} \\
d_{\mathrm{GH}}(n^{-1}\cdot f_{\max},(1-\nu_\mu)\cdot C^1) & \xrightarrow[n \rightarrow \infty]{\enskip (\mathbb{P}) \enskip} 0 \label{conv4} 
\end{align}
and the result then follows from the triangle inequality. 

We start by proving \eqref{conv1}.  By Lemma \ref{l:GHestimate}
\begin{align} \label{eq:est1}
d_{\mathrm{GH}}(\mathcal{SD}_n^w, \mathcal{D}^1_n) \leq \diam(\bar{\mathcal{D}}^1_n) \leq |\bar{\mathcal{D}}_n^1|,
\end{align}
where $|G|$ denotes the number of vertices in the graph $G$. Since $\mathcal{Z}^\mu_\varnothing(R_\mu) < 1$, it follows from \eqref{subcritical_scheme} and a general result on Gibbs-partitions with a subcritical composition scheme \cite[Theorem 3.1]{stufler-gibbs} that  $|\bar{\mathcal{D}}_n^1|$ converges in total variation towards a finite random variable which by \eqref{eq:est1} concludes the proof of \eqref{conv1}. It also follows that $\frac{|\mathcal{D}^1_n|}{n} \to 1$ in probability.
 
By Lemma \ref{l:GHestimate}
 \begin{align} \label{eq:est2}
d_{\mathrm{GH}}(\mathcal{D}_n^1, \mathcal{D}^2_n) \leq \diam(\bar{\mathcal{D}}^2_n) \leq |\bar{\mathcal{D}}_n^2|.
 \end{align}
Let $\hat \mu$ denote the probability distribution with mean $\nu_\mu$ obtained by tilting the weights $\mu$ and choose $\hat \mu_\varnothing$ by simultaneously tilting $\mu_\varnothing$ according to \eqref{eq:mus}. Again, using  \cite[Theorem 3.1]{stufler-gibbs} it holds that conditionally on $|\mathcal{D}_n^1|$, $\mathcal{D}_n^1$ is distributed as a $w$-weighted NCD and thus the shape of its weak dual is distributed as $\GW_{|\mathcal{D}^1_n|}^{\hat \mu,\hat \mu_{\varnothing}}$. Therefore, since $\frac{|\mathcal{D}_n^1|}{n} \to 1$ in probability, it holds by Lemma \ref{l:cond} that $|\bar{\mathcal{D}}_n^2|$ converges towards a finite random variable when $n\to \infty$ which along with \eqref{eq:est2} proves \eqref{conv2}. It also follows that $\frac{|\mathcal{D}^2_n|}{n} \to 1$ in probability.

Applying Lemma \ref{l:cond} again allows us to conclude that conditionally on $|\mathcal{D}_n^2|$ , $\mathcal{D}_n^2$ is distributed like a NCD whose weak dual has a shape which is a one type Galton-Watson tree $\GW_{|\mathcal{D}^2_n|}^{\hat \mu}$. Let us now decompose $\mathcal{D}^2_n$ into the face $f_{\max}$ and the dissections which share an edge with this face. Denote these dissections by  $(\mathcal{E}_i)_{i=1}^{\deg(f_{\max})}$ in clockwise order from the edge which belonged to $f_0$, see Fig.~\ref{f:condensation}. By a straigtforward generalization of Lemma \ref{l:GHestimate}
\begin{align} \nonumber
	d_{\mathrm{GH}}(\mathcal{D}_n^2,f_{\max}) &\leq \max \{\diam(\mathcal{E}_i)~:~ 1 \leq i \leq \deg(f_{\max})\} \\
	&\leq \max \{|\mathcal{E}_i|~:~ 1 \leq i \leq \deg(f_{\max})\}. \label{eq:est5}
\end{align}
It follows from the bijection with trees and the fact that $\nu_\mu < 1$ that there is a slowly varying sequence $(\tilde L_n)_{n\geq 1}$ such that
\begin{align} \label{eq:est3}
		\max \{|\mathcal{E}_i|~:~ 1\leq i \leq \deg(f_{\max})\} = O_p(\tilde L_n n^{1/\alpha}),
\end{align}
and 
\begin{align} \label{eq:est4}
	\deg(f_{\max}) = (1-\nu_\mu)n + O_p(\tilde L_n n^{1/\alpha})
\end{align}
see \cite[Theorem 1 and Corollary 1]{kortchemski}, where we have used $|\mathcal{D}^2_n|/n \to 1$ in probability to replace $|\mathcal{D}_n^2|$ by $n$. The estimates \eqref{eq:est5} and \eqref{eq:est4} prove \eqref{conv3}.

Finally we prove \eqref{conv4}. We construct a correspondence between $f_{\max}$ and $C^1$ by dividing the latter into $\deg(f_{\max})$ number of arcs of equal length and letting subsequent points of $f_{\max}$ correspond to subsequent arcs of $C^1$. By bounding its distortion we find that
\begin{align*}
	d_{\mathrm{GH}}(n^{-1} \cdot f_{\max},(1-\nu_\mu)\cdot C^1) \leq \left|\frac{\deg(f_{\max})}{n}-(1-\nu_\mu)\right| + O(\deg(f_{\max})^{-1}) \to 0
\end{align*}
in probability as $n\to \infty$ by \eqref{eq:est4} which proves \eqref{conv4}.

%

\end{proof}

\subsection*{Acknowledgement} We thank Thomas Vallier and Thomas Selig for discussions in an early stage of this work. We are grateful for valuable comments from Benedikt Stufler on a revised version of the paper. The research is partially supported by the Icelandic Research Fund, grant number 185233-051.

\end{document}